\documentclass[]{article}

\usepackage{graphicx}
\usepackage{amsmath}
\usepackage{amssymb}
\usepackage{amsthm}
\usepackage{pxfonts}
\usepackage{enumerate}
\usepackage{color}
\usepackage{mathdots}
\usepackage{sectsty}
\usepackage[hidelinks]{hyperref}
\usepackage{tikz}
\usepackage{caption}
\usepackage{adjustbox}

\sectionfont{\scshape\centering\fontsize{11}{14}\selectfont}
\subsectionfont{\scshape\fontsize{11}{14}\selectfont}
\usepackage{fancyhdr}

\newcommand\shorttitle{Infinite p-adic random matrices and ergodic decomposition of p-adic Hua measures}
\newcommand\authors{Theodoros Assiotis}

\newtheorem{thm}{Theorem}[section]

\newtheorem{lem}[thm]{Lemma}
\newtheorem{defn}[thm]{Definition}
\newtheorem{rmk}[thm]{Remark}
\newtheorem{prop}[thm]{Proposition}

\title{\large \bf INFINITE P-ADIC RANDOM MATRICES AND ERGODIC DECOMPOSITION OF P-ADIC HUA MEASURES}
\author{\small THEODOROS ASSIOTIS}
\date{}

\begin{document}

\maketitle

\begin{abstract}
Neretin in \cite{NeretinHuaMeasures} constructed an analogue of the Hua measures on the infinite $p$-adic matrices $\textnormal{Mat}\left(\mathbb{N},\mathbb{Q}_p\right)$. Bufetov and Qiu in \cite{BufetovQiuClassificaiton} classified the ergodic measures on $\textnormal{Mat}\left(\mathbb{N},\mathbb{Q}_p\right)$ that are invariant under the natural action of $\textnormal{GL}(\infty,\mathbb{Z}_p)\times \textnormal{GL}(\infty,\mathbb{Z}_p)$. In this paper we solve the problem of ergodic decomposition for the $p$-adic Hua measures introduced by Neretin. We prove that the probability measure governing the ergodic decomposition has an explicit expression which identifies it with a Hall-Littlewood measure on partitions. Our arguments involve certain Markov chains.
\end{abstract}

\tableofcontents

\section{Introduction}

\subsection{Setting and history of the problem}
In the last few decades, especially after the introduction of Vershik's ergodic method \cite{VershikErgodic} and the seminal works of Vershik and Kerov on the indecomposable characters of the infinite dimensional unitary \cite{VershikKerovUnitary} and infinite symmetric \cite{VershikKerovSymmetric} groups, there has been considerable interest both in classifying indecomposable characters of such inductive limit groups and also in the allied problem of classification of ergodic measures invariant under the action of certain infinite dimensional groups, see for example \cite{OlshanskiVershik}, \cite{Pickrell}, \cite{Defosseux}, \cite{GKV}, \cite{BufetovQiuClassificaiton}, \cite{AssiotisNajnudel} for more details. Once the task of classification is completed, a natural direction of research, also known as the problem of harmonic analysis, see \cite{OlshanskiHarmonic}, is to describe explicitly how a distinguished reducible character, or respectively an invariant measure, decomposes into extremal objects; namely either indecomposable characters or ergodic measures, see for example \cite{BorodinOlshanskiErgodic}, \cite{BorodinOlshanskiHarmonic}, \cite{GorinOlshanski}, \cite{BufetovI}, \cite{BufetovII}, \cite{BufetovIII}, \cite{Qiu}, \cite{BufetovQiu}, \cite{Cuenca}, \cite{AssiotisInverseWishart}.

The purpose of this article is to solve such a problem in a setting that had not been considered previously. More precisely, we are interested in probability measures on the space of infinite $p$-adic matrices $\textnormal{Mat}\left(\mathbb{N},\mathbb{Q}_p\right)$ invariant under the natural action of the group $\textnormal{GL}(\infty,\mathbb{Z}_p)\times \textnormal{GL}(\infty,\mathbb{Z}_p)$\footnote{All these notions will be defined properly in the sequel.}. The ergodic measures for this action were classified recently by Bufetov and Qiu in \cite{BufetovQiuClassificaiton}. A few years earlier, in \cite{NeretinHuaMeasures}, Neretin had already constructed on $\textnormal{Mat}\left(\mathbb{N},\mathbb{Q}_p\right)$ a $p$-adic analogue of the celebrated Hua measures, see \cite{Hua}, \cite{Pickrellmeasure}, \cite{Neretin}. The main result of this paper is an explicit description of how the $p$-adic Hua measures of Neretin \cite{NeretinHuaMeasures} decompose into the ergodic measures of Bufetov and Qiu from \cite{BufetovQiuClassificaiton}. Corresponding problems over the complex numbers (instead of the $p$-adics) were solved by Borodin, Olshanski, Bufetov and Qiu, see \cite{BorodinOlshanskiErgodic}, \cite{BorodinOlshanskiHarmonic}, \cite{BufetovI}, \cite{BufetovII}, \cite{BufetovIII}, \cite{Qiu}, \cite{BufetovQiu} for more details.

The scheme of proof of our main result follows an adaptation of the Vershik-Kerov ergodic method \cite{VershikErgodic} of approximation from finite dimensions. This of course does not come as a surprise as this strategy is in fact shared by all the papers mentioned above. Essentially, the actual problem that we solve and thus the main contribution of this paper is taking the large $N$-limit of the analogue of the singular values (we will be more precise later) of certain $N\times N$ $p$-adic random matrices. This is where our methods differ from the ones used in \cite{BorodinOlshanskiErgodic}, \cite{BorodinOlshanskiHarmonic}, \cite{GorinOlshanski}, \cite{BufetovI}, \cite{BufetovII}, \cite{BufetovIII}, \cite{Qiu}, \cite{BufetovQiu}, \cite{Cuenca}, \cite{AssiotisInverseWishart}.

In all previous works just mentioned\footnote{There are of course cases of problems of harmonic analysis which do not involve determinantal point processes and orthogonal polynomials, see for example \cite{KOV}. We note that the methods used there are also different from ours.} \cite{BorodinOlshanskiErgodic}, \cite{BorodinOlshanskiHarmonic}, \cite{GorinOlshanski}, \cite{BufetovI}, \cite{BufetovII}, \cite{BufetovIII}, \cite{Qiu}, \cite{BufetovQiu}, \cite{Cuenca}, \cite{AssiotisInverseWishart} one uses the technology of determinantal point processes \cite{Soshnikov} to take this large $N$ limit. The task then boils down to establishing asymptotics and estimates for certain families of orthogonal polynomials, which can be technically rather challenging. The technique of determinantal point processes does not seem applicable in our setting and we need to take a different approach.

Instead, we find some discrete time Markov chain structure underlying the law of the singular values. This is inspired by the work of Evans \cite{Evans} who studied the volume measure on $\textnormal{Mat}\left(N,\mathbb{Z}_p\right)$ (which turns out to give rise to an ergodic measure, see Section \ref{ProofSection} for more details) and an extensive series of works \cite{FulmanConjugacyClasses}, \cite{FulmanRMT}, \cite{FulmanRogersRamanujan1}, \cite{FulmanRogersRamanujan2}, \cite{FulmanCohenLenstra} by Fulman in the different setting of random matrices over a finite field. These Markov chains appear in a number of different settings, in particular they are closely related to the Rogers-Ramanujan identities and the Cohen-Lenstra heuristics from number theory, see for example \cite{FulmanRogersRamanujan1}, \cite{FulmanRogersRamanujan2}, \cite{FulmanCohenLenstra}.

A novel feature of the present paper compared to the simpler setting studied in \cite{Evans} is the fact that in order to sample the law of the singular values one needs to run two discrete time Markov chains with coupled initial conditions, or equivalently (and possibly more intuitively) sample an initial condition and run two Markov chains one moving forwards and the other one backwards in discrete time. 

Once we have this Markov chain representation, the task of taking the $N\to \infty$ limit becomes, from a technical standpoint, rather straightforward. In the limit we obtain a Markov chain, that also appears in the work of Fulman \cite{FulmanConjugacyClasses}, \cite{FulmanRMT}, \cite{FulmanRogersRamanujan1} in the context of random matrix theory over a finite field, which allows one to sample according to a Hall-Littlewood measure on partitions \cite{FulmanConjugacyClasses}, \cite{MacdonaldProcesses}. Finally, we believe that a variation of this approach could possibly be used for other $p$-adic random matrices, in particular in the more complicated case of the $p$-adic Hua measures on symmetric matrices $\textnormal{Sym}\left(\mathbb{N},\mathbb{Q}_p\right)$ and we say more about this in Section \ref{FurtherDirectionsSection}.

In the rest of the introduction we give some background and state the main result of this paper, Theorem \ref{MainResult} below, precisely.
\subsection{Preliminaries}
In order to make this paper self-contained we begin with a number of rather standard preliminaries. The reader is referred to \cite{BufetovQiuClassificaiton} and \cite{NeretinHuaMeasures} and the references therein for more details.

Let $p$ be a prime. Any rational $r\in \mathbb{Q}\backslash\{0\}$ can be written as $r=p^u\frac{a}{b}$ where $a,b$ are not divisible by $p$. Set $|r|=p^{-u}$ and $|0|=0$. Then, the map $(x,y)\mapsto |x-y|$ defines a metric on $\mathbb{Q}$ and we denote the completion of $\mathbb{Q}$ in this metric by $\mathbb{Q}_p$ which is the field of $p$-adic numbers. The closed unit ball around $0$, $\mathbb{Z}_p=\{x\in \mathbb{Q}_p: |x|\le 1 \}$ is called the ring of $p$-adic integers.

Let $\textnormal{GL}(N,\mathbb{Q}_p)$ and $\textnormal{GL}(N,\mathbb{Z}_p)$ be the groups of invertible $N\times N$ matrices over $\mathbb{Q}_p$ and $\mathbb{Z}_p$ respectively. Then, we define the inductive limit group:
\begin{align*}
\textnormal{GL}(\infty,\mathbb{Z}_p)=\underset{\rightarrow}{\lim}\textnormal{GL}(N,\mathbb{Z}_p)
\end{align*}
where the limit is taken with respect to the maps:
\begin{align*}
\textnormal{GL}(N,\mathbb{Z}_p)\ni\mathbf{B}\mapsto \begin{bmatrix}
\mathbf{B} &0\\
 0 & 1
\end{bmatrix} \in \textnormal{GL}(N+1,\mathbb{Z}_p).
\end{align*}
Equivalently, $\textnormal{GL}(\infty,\mathbb{Z}_p)$ is the group of infinite invertible matrices $\mathbf{B}=\{\mathbf{B}_{ij} \}_{i,j=1}^\infty$ over $\mathbb{Z}_p$ so that $\mathbf{B}_{ij}=\mathbf{1}(i=j)$ if $i+j$ is large enough. Moreover, we define the space of infinite matrices over $\mathbb{Q}_p$:
\begin{align*}
\textnormal{Mat}\left(\mathbb{N},\mathbb{Q}_p\right)=\big\{\mathbf{Z}=\{\mathbf{Z}_{ij}\}_{i,j=1}^{\infty}: \mathbf{Z}_{ij}\in \mathbb{Q}_p \big\}.
\end{align*}
In this paper we are interested in the following group action of $\textnormal{GL}(\infty,\mathbb{Z}_p)\times \textnormal{GL}(\infty,\mathbb{Z}_p)$ on $\textnormal{Mat}\left(\mathbb{N},\mathbb{Q}_p\right)$:
\begin{align*}
\left(\left(\mathbf{B}_1,\mathbf{B}_2\right),\mathbf{Z}\right)\mapsto\mathbf{B}_1\mathbf{Z}\mathbf{B}_2^{-1}, \ \ \mathbf{B}_1,\mathbf{B}_2 \in \textnormal{GL}(\infty,\mathbb{Z}_p),\  \mathbf{Z}\in \textnormal{Mat}\left(\mathbb{N},\mathbb{Q}_p\right).
\end{align*}
Let $\mathcal{P}\left(\textnormal{Mat}\left(\mathbb{N},\mathbb{Q}_p\right)\right)$ be the space of probability measures on $\textnormal{Mat}\left(\mathbb{N},\mathbb{Q}_p\right)$ endowed with the weak topology. We let $\mathcal{P}_{\textnormal{inv}}\left(\textnormal{Mat}\left(\mathbb{N},\mathbb{Q}_p\right)\right)$ be the convex subset of probability measures invariant under the action of $\textnormal{GL}(\infty,\mathbb{Z}_p)\times \textnormal{GL}(\infty,\mathbb{Z}_p)$ defined above. Moreover, we denote by $\mathcal{P}_{\textnormal{erg}}\left(\textnormal{Mat}\left(\mathbb{N},\mathbb{Q}_p\right)\right)$ the subset of ergodic measures, namely $\mathcal{M}\in\mathcal{P}_{\textnormal{erg}}\left(\textnormal{Mat}\left(\mathbb{N},\mathbb{Q}_p\right)\right)$ if $\mathcal{M}\in\mathcal{P}_{\textnormal{inv}}\left(\textnormal{Mat}\left(\mathbb{N},\mathbb{Q}_p\right)\right)$ and for any $\textnormal{GL}(\infty,\mathbb{Z}_p)\times \textnormal{GL}(\infty,\mathbb{Z}_p)$-invariant Borel subset $\mathcal{A}\subset\textnormal{Mat}\left(\mathbb{N},\mathbb{Q}_p\right)$, either $\mathcal{M}\left(\mathcal{A}\right)=0$ or $\mathcal{M}\left(\textnormal{Mat}\left(\mathbb{N},\mathbb{Q}_p\right)\backslash\mathcal{A}\right)=0$. Again, $\mathcal{P}_{\textnormal{erg}}\left(\textnormal{Mat}\left(\mathbb{N},\mathbb{Q}_p\right)\right)$ is endowed with the induced weak topology. Finally, we note that by a general result of Bufetov, see Proposition 2 in \cite{Bufetov}, the notion of ergodicity coincides with the notion of indecomposability, namely $\mathcal{M}\in\mathcal{P}_{\textnormal{inv}}\left(\textnormal{Mat}\left(\mathbb{N},\mathbb{Q}_p\right)\right)$ is indecomposable if the equality $\mathcal{M}=c \mathcal{M}_1+(1-c)\mathcal{M}_2$, with $c \in (0,1)$, $\mathcal{M}_1, \mathcal{M}_2 \in\mathcal{P}_{\textnormal{inv}}\left(\textnormal{Mat}\left(\mathbb{N},\mathbb{Q}_p\right)\right)$ implies $\mathcal{M}=\mathcal{M}_1=\mathcal{M}_2$.

For any $m\in \mathbb{N}$ we denote by $d\mathsf{vol}=d\mathsf{vol}_m$ the Haar measure on $\mathbb{Q}_p^m$ normalized by $\mathsf{vol}\left(\mathbb{Z}_p^m\right)=1$. Using the identification $\textnormal{Mat}\left(N,\mathbb{Q}_p\right)\simeq \mathbb{Q}_p^{N^2}$ we can define $\mathsf{vol}$ on $\textnormal{Mat}\left(N,\mathbb{Q}_p\right)$ with the normalization $\mathsf{vol}\left(\textnormal{Mat}\left(N,\mathbb{Z}_p\right)\right)=1$. Note that, the measure $\mathsf{vol}$ on $\textnormal{Mat}\left(N,\mathbb{Q}_p\right)$ is invariant under the action of $\textnormal{GL}(N,\mathbb{Z}_p)\times \textnormal{GL}(N,\mathbb{Z}_p)$ on $\textnormal{Mat}\left(N,\mathbb{Q}_p\right)$ considered above. Finally, the Haar measure $d\mathsf{Haar}$ on $\textnormal{GL}(N,\mathbb{Q}_p)$, normalized so that it is a probability measure on the compact group $\textnormal{GL}(N,\mathbb{Z}_p)$, is given by:
\begin{align}
d\mathsf{Haar}\left(\mathbf{Z}\right)=\frac{1}{\left(p^{-1};p^{-1}\right)_N}|\det\left(\mathbf{Z}\right)|^{-N} d\mathsf{vol}\left(\mathbf{Z}\right),
\end{align}
where here and throughout the paper, $\left(\alpha;q\right)_N$ denotes the $q$-Pochhammer symbol:
\begin{align*}
\left(\alpha;q\right)_N=\prod_{j=0}^{N-1}\left(1-\alpha q^j\right), \ \left(\alpha;q\right)_0=1.
\end{align*}

\subsection{Classification of ergodic measures}

The goal of this section is to recall the classification of ergodic measures $\mathcal{P}_{\textnormal{erg}}\left(\textnormal{Mat}\left(\mathbb{N},\mathbb{Q}_p\right)\right)$ due to Bufetov and Qiu \cite{BufetovQiuClassificaiton}. We begin with some notation and definitions. We consider the following space:
\begin{align}
\Delta=\big\{\mathbf{k}=\left(k_j\right)_{j=1}^{\infty}\in \left(\mathbb{Z}\cup \{-\infty \}\right)^{\mathbb{N}}:k_1\ge k_2\ge k_3\ge \cdots \big\}.
\end{align}
We endow $\Delta$ with the induced topology of Tychonoff's product topology on $\left(\mathbb{Z}\cup \{-\infty \}\right)^{\mathbb{N}}$.
\begin{defn}\label{DefinitionErgodicMeasure}
We let
\begin{align*}
X_i^{(N)}, \ Y_i^{(N)}, \ Z_{ij}, \ i,j,N=1,2,3\dots
\end{align*}
be independent random variables identically distributed according to the Haar probability measure on $\mathbb{Z}_p$, namely, $d\mathsf{vol}_1$. Then, given $\mathbf{k}\in \Delta$ we let:
\begin{align*}
\mu_{\mathbf{k}}=\mathsf{Law}\left(\mathbf{A}_{\mathbf{k}}\right)
\end{align*}
to be the law of the infinite random matrix $\mathbf{A}_{\mathbf{k}} \in \textnormal{Mat}\left(\mathbb{N},\mathbb{Q}_p\right)$ defined as follows\footnote{We note that the series in (\ref{DefInfiniteMatrix}) defining each entry $(i,j)$ of $\mathbf{A}_{\mathbf{k}}$ converges almost surely, which can be seen as follows. First, observe that if $k=\lim_{j\to \infty}k_j \neq -\infty$ then the series has a finite number of terms. On the other hand, if $k=-\infty$ (in which case the series can have a finite or infinite number of terms) then by the ultrametric property $|x+y|\le \max \{|x|,|y|\}$ it is easily seen (recall that $|p^{-k_N}|=p^{k_N}$) that the partial sums of the series in (\ref{DefInfiniteMatrix}) almost surely form a Cauchy sequence.}:
\begin{align}\label{DefInfiniteMatrix}
\mathbf{A}_{\mathbf{k}}=\left[\sum_{N:k_N>k}^{}p^{-k_N}X_i^{(N)}Y_j^{(N)}+p^{-k}Z_{ij}\right]_{i,j\in \mathbb{N}},
\end{align}
where $k=\lim_{j\to \infty}k_j\in \mathbb{Z}\cup \{-\infty \}$.
\end{defn}

We are now in a position to state the classification theorem of Bufetov and Qiu in \cite{BufetovQiuClassificaiton}.

\begin{thm}\label{ClassificationTheorem}
The map $\mathbf{k}\mapsto \mu_{\mathbf{k}}$ is a homeomorphism between $\Delta$ and $\mathcal{P}_{\textnormal{erg}}\left(\textnormal{Mat}\left(\mathbb{N},\mathbb{Q}_p\right)\right)$.
\end{thm}

\subsection{$p$-adic Hua measures}
In order to introduce the $p$-adic Hua measures in Definition \ref{DefinitionPadicHua} below we will need the following well-known result on decomposing a matrix $\mathbf{Z} \in \textnormal{Mat}\left(N,\mathbb{Q}_p\right)$, which will also be important in the subsequent analysis, see for example \cite{BufetovQiuClassificaiton}, \cite{NeretinHuaMeasures} and the references therein.
\begin{lem}
Let $\mathbf{Z} \in \textnormal{Mat}\left(N,\mathbb{Q}_p\right)$. Then, $\mathbf{Z}$ can be written as 
\begin{align}\label{SingularDecomposition}
\mathbf{Z}=\mathbf{B}\cdot\textnormal{diag}\left(p^{-k_1},p^{-k_2},\dots,p^{-k_N}\right)\cdot \mathbf{C}, \ \textnormal{ with }\mathbf{B}, \mathbf{C}\in \textnormal{GL}\left(N,\mathbb{Z}_p\right)
\end{align}
where $k_1\ge k_2\ge k_3\ge \cdots \ge k_N\ge -\infty$ and $\textnormal{diag}\left(p^{-k_1},p^{-k_2},\dots,p^{-k_N}\right)$ denotes the diagonal matrix with diagonal elements $p^{-k_1},p^{-k_2},\dots,p^{-k_N}$. The $N$-tuple $\left(k_1,k_2,\dots,k_N\right)$ is uniquely determined by $\mathbf{Z}$ and we call these the singular numbers of $\mathbf{Z}$.
\end{lem}

\begin{defn}\label{DefinitionPadicHua}
Consider the following function $\gamma(\cdot)$ on $\textnormal{Mat}\left(N,\mathbb{Q}_p\right)$, for $\mathbf{Z} \in \textnormal{Mat}\left(N,\mathbb{Q}_p\right)$ with decomposition as in (\ref{SingularDecomposition}), by the formula:
\begin{align*}
\gamma(\mathbf{Z})=\prod_{k_j>0}^{}p^{k_j},
\end{align*}
with the understanding that $\gamma(\mathbf{Z})\equiv 1$ if $k_j\le 0$ for all $j$. Let $s>-1$. We define the $p$-adic Hua probability measures $\mathsf{M}_N^{(s)}$ on $\textnormal{Mat}\left(N,\mathbb{Q}_p\right)$\footnote{Implicitly, by its very definition in \cite{NeretinHuaMeasures}, see the computation of the normalization constant in Section 2 therein, the probability measure $\mathsf{M}_N^{(s)}$ is actually supported on $\textnormal{GL}(N,\mathbb{Q}_p)$.} by the formula, see \cite{NeretinHuaMeasures}:
\begin{align}\label{HuaMeasuresDefinition}
d\mathsf{M}_N^{(s)}(\mathbf{Z})=\frac{\left(p^{-1-s};p^{-1}\right)^2_{N}}{\left(p^{-1-s};p^{-1}\right)_{2N}}\gamma(\mathbf{Z})^{-s-2N}d\mathsf{vol}\left(\mathbf{Z}\right).
\end{align}
\end{defn}
Observe that, for all $N\ge 1$, $\mathsf{M}_N^{(s)}$ is $\textnormal{GL}(N,\mathbb{Z}_p)\times \textnormal{GL}(N,\mathbb{Z}_p)$-invariant. For any $N\ge 1$, we define the corners maps $\Pi^{N+1}_N:\textnormal{Mat}\left(N+1,\mathbb{Q}_p\right) \longrightarrow \textnormal{Mat}\left(N,\mathbb{Q}_p\right)$ by:
\begin{align*}
\Pi_N^{N+1}\left(\{\mathbf{Z}_{ij}\}_{i,j=1}^{N+1} \right)= \{\mathbf{Z}_{ij} \}_{i,j=1}^N.
\end{align*}
Note that, equivalently we could have defined the space of infinite matrices $\textnormal{Mat}\left(\mathbb{N},\mathbb{Q}_p\right)$ as the projective limit $\underset{\leftarrow}{\lim}\textnormal{Mat}\left(N,\mathbb{Q}_p\right)$ under the maps $\Pi_N^{N+1}$. Finally, for any $N\ge 1$, we also define $\Pi^{\infty}_N:\textnormal{Mat}\left(\mathbb{N},\mathbb{Q}_p\right) \longrightarrow \textnormal{Mat}\left(N,\mathbb{Q}_p\right)$ by:
\begin{align*}
\Pi_N^{\infty}\left(\{\mathbf{Z}_{ij}\}_{i,j=1}^{\infty} \right)= \{\mathbf{Z}_{ij} \}_{i,j=1}^N.
\end{align*}

The $p$-adic Hua measures have the remarkable property that they are consistent under the corners maps, which is the content of the following result due to Neretin \cite{Neretin}:
\begin{prop}\label{ConstructionInfiniteMeasure}
Let $s>-1$. Then, the $p$-adic Hua measures are consistent with respect to the corners maps:
\begin{align*}
\left(\Pi_{N}^{N+1}\right)_*\mathsf{M}_{N+1}^{(s)}=\mathsf{M}_N^{(s)}, \ \forall N\ge 1.
\end{align*}
Thus, by Kolmogorov's theorem we obtain a unique $\textnormal{GL}(\infty,\mathbb{Z}_p)\times \textnormal{GL}(\infty,\mathbb{Z}_p)$-invariant probability measure $\mathsf{M}^{(s)}$ on $\textnormal{Mat}\left(\mathbb{N},\mathbb{Q}_p\right)$, namely $\mathsf{M}^{(s)}\in \mathcal{P}_{\textnormal{inv}}\left(\textnormal{Mat}\left(\mathbb{N},\mathbb{Q}_p\right)\right)$, so that:
\begin{align*}
\left(\Pi_{N}^{\infty}\right)_*\mathsf{M}^{(s)}=\mathsf{M}_N^{(s)}, \ \forall N\ge 1.
\end{align*}
\end{prop}

Finally, by combining the general results of Bufetov from \cite{Bufetov} for actions of inductively compact groups (thus applicable to the current setting) along with Theorem \ref{ClassificationTheorem} we obtain the following proposition, which gives the existence of a unique probability measure describing the ergodic decomposition of $\mathsf{M}^{(s)}$:

\begin{prop}\label{ErgodicDecompositionIntroduction}
Let $s>-1$. Then, there exists a unique probability measure $\nu^{(s)}$ on $\Delta$ such that:
\begin{align}
\mathsf{M}^{(s)}=\int_{\Delta}^{}\mu_{\mathbf{k}}\nu^{(s)}\left(d\mathbf{k}\right).
\end{align}
\end{prop}

The main result of this paper, Theorem \ref{MainResult} below, is an explicit description of the measure $\nu^{(s)}$.

\subsection{Main result}

We first introduce for any $s>-1$ a certain Markov kernel $\mathsf{P}^{(s)}$ and probability measure $\pi^{(s)}$ on non-negative integers. The proof of the fact that $\mathsf{P}^{(s)}(x,\cdot)$ and $\pi^{(s)}(\cdot)$ are probability measures, namely that they correctly sum (over the non-negative integers) up to $1$, can be found in \cite{FulmanRogersRamanujan2}.
\begin{defn}\label{MainDefinition} Let $s>-1$. We use the notations $\mathbb{Z}_+=\{0,1,2,3,\dots\}$, $\llbracket0,x \rrbracket=\{0,1,\dots,x\}$ and $\mathbf{1}_{\mathcal{A}}$ for the indicator function of a set $\mathcal{A}$. We then define the following Markov kernel $\mathsf{P}^{(s)}$ on $\mathbb{Z}_+$ given by the formula, for $x_1,x_2\in \mathbb{Z}_+$:
\begin{align}\label{DefMarkovKernel}
\mathsf{P}^{(s)}\left(x_1,x_2\right)=\frac{p^{-x_2^2-sx_2}\left(p^{-1};p^{-1}\right)_{x_1}\left(p^{-1-s};p^{-1}\right)_{x_1}}{\left(p^{-1};p^{-1}\right)_{x_2}\left(p^{-1};p^{-1}\right)_{x_1-x_2}\left(p^{-1-s};p^{-1}\right)_{x_2}}\mathbf{1}_{\llbracket0,x_1 \rrbracket}(x_2).
\end{align}
For the distinguished case $s=0$ we simply write $\mathsf{P}=\mathsf{P}^{(0)}$. Moreover, define the probability measure $\pi^{(s)}$ on $\mathbb{Z}_+$ by the formula:
\begin{align*}
\pi^{(s)}(x)=\frac{p^{-x^2-sx}\left(p^{-1-s};p^{-1}\right)_\infty}{\left(p^{-1};p^{-1}\right)_x\left(p^{-1-s};p^{-1}\right)_x} , \ x\in \mathbb{Z}_+.
\end{align*}
\end{defn}
In the sequel we will also need the following subset of $\Delta$:
\begin{align*}
\Delta_{0}=\big\{\mathbf{k}=\left(k_j\right)_{j=1}^{\infty}\in \mathbb{Z}_+^{\mathbb{N}}:k_1\ge k_2\ge k_3 \ge \cdots; k_i\equiv 0 \textnormal{ for i large enough}\big\}.
\end{align*}
Finally, for each $i\in \mathbb{Z}\cup \{-\infty \}$ we define $l_i\left(\mathbf{k}\right)=\#\big\{j:k_j=i \big\}$. 

We now arrive at our main result which gives a complete explicit description of the probability measure $\nu^{(s)}$.

\begin{thm}\label{MainResult}
Let $s>-1$. Then, the probability measure $\nu^{(s)}$ is supported on $\Delta_0$. Moreover, a random $\mathbf{k}$ with $\mathsf{Law}\left(\mathbf{k}\right)=\nu^{(s)}$ can be sampled as follows, see also Figure \ref{Figure1} for an illustration:
\begin{enumerate}
\item Sample $\sum_{i=1}^{\infty}l_i(\mathbf{k})$ according to $\pi^{(s)}$.
\item To then sample $\sum_{i=2}^{\infty}l_i(\mathbf{k}), \ \sum_{i=3}^{\infty}l_i(\mathbf{k}), \ \sum_{i=4}^{\infty}l_i(\mathbf{k}), \dots$ run a discrete time Markov chain with transition kernel $\mathsf{P}^{(s)}$ starting from $\sum_{i=1}^{\infty}l_i(\mathbf{k})$; namely the conditional distribution of $\sum_{j=i+1}^{\infty}l_j(\mathbf{k})$ given $\sum_{j=i}^{\infty}l_j(\mathbf{k})$ is given by the probability measure $\mathsf{P}^{(s)}\left(\sum_{j=i}^{\infty}l_j(\mathbf{k}),\cdot\right)$.
\end{enumerate}
More precisely, the sequence $\{\mathsf{X}_i(\mathbf{k}) \}_{i=1}^{\infty}$ where $\mathsf{X}_i(\mathbf{k})=\sum_{j=i}^{\infty}l_j(\mathbf{k})$ forms a time homogeneous Markov chain on $\mathbb{Z}_+$ with initial distribution $\pi^{(s)}$ and transition kernel $\mathsf{P}^{(s)}$. Thus, since $\mathbf{k}\in \Delta_0$ is completely determined by $\{l_i(\mathbf{k})\}_{i=1}^{\infty}$, which is in turn determined by $\{\sum_{j=i}^{\infty}l_j(\mathbf{k})\}_{i=1}^{\infty}$, we can write $\nu^{(s)}$ explicitly:
\begin{align}
\nu^{(s)}(\mathbf{k})=\left(p^{-1-s};p^{-1}\right)_\infty p^{-\sum_{i=1}^{\infty}\left(\sum_{j=i}^{\infty}l_j(\mathbf{k})\right)^2-s\sum_{i=1}^{\infty}il_i(\mathbf{k})} \prod_{i=1}^{\infty}\frac{1}{\left(p^{-1};p^{-1}\right)_{l_i(\mathbf{k})}}, \ \ \mathbf{k}\in \Delta_0.
\end{align}
\end{thm}

\begin{figure}
\captionsetup{singlelinecheck = false, justification=justified}
\begin{tikzpicture}
               \draw[->,thick] (1,0) to [out=45,in=135] (2,0);
               
              \node[above] at (1.5,0.2) {$\pi^{(s)}$};
              \node[] at (3,0) {$\sum_{i=1}^{\infty}l_i(\mathbf{k})$};
              
                             \draw[->,thick] (4,0) to [out=45,in=135] (5,0);
                            \node[above] at (4.5,0.2) {$\mathsf{P}^{(s)}$};               
              \node[] at (6,0) {$\sum_{i=2}^{\infty}l_i(\mathbf{k})$};
              
              \draw[->,thick] (7,0) to [out=45,in=135] (8,0);
                                          \node[above] at (7.5,0.2) {$\mathsf{P}^{(s)}$};               
              \node[] at (9,0) {$\sum_{i=3}^{\infty}l_i(\mathbf{k})$};
              
                            \draw[->,thick] (10,0) to [out=45,in=135] (11,0);
                                         \node[above] at (10.5,0.2) {$\mathsf{P}^{(s)}$};                              
             \node[] at (12,0) {$\sum_{i=4}^{\infty}l_i(\mathbf{k})$};
             
                                         \draw[->,thick] (13,0) to [out=45,in=135] (14,0);
                                        \node[above] at (13.5,0.2) {$\mathsf{P}^{(s)}$};                                            
              \node[] at (14.5,0) {$\dots$};

 \end{tikzpicture}
 \caption{An illustration of the Markov chain from Theorem \ref{MainResult}.}\label{Figure1}
 \end{figure}
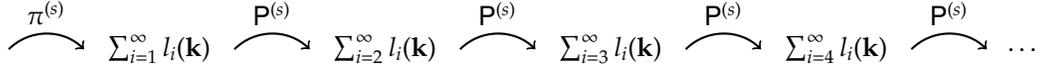

We conclude this section with a series of observations and remarks on consequences, connections and possible extensions of our main result. Observe that, from Theorem \ref{MainResult} we have:
\begin{align*}
\mathsf{Law}\left(\#\{i:k_i>0\}\right)=\pi^{(s)}
\end{align*}
and moreover if we consider the stopping time $\tau_0=\inf \{j: \mathsf{X}_j(\mathbf{k})=0 \}$
then
\begin{align*}
\mathsf{Law}\left(k_1\right)=\mathsf{Law}\left(\tau_0-1\right).
\end{align*}

\begin{rmk}
The results that follow in this paper can be transferred with mostly notational modifications to the more general setting of non-discrete non-Archimedean locally compact fields. The main correspondences are as follows. $\mathbb{Q}_p \leadsto F$, a non-discrete non-Archimedean locally compact field with absolute value $|\cdot|$. $\mathbb{Z}_p \leadsto \mathcal{O}_F=\{x \in F:|x|\le 1 \}$, the ring of integers of $F$. $p\leadsto \omega$, any generator of the ideal $\{x\in F: |x|<1 \}$. We have decided to restrict our attention to the $p$-adics $\mathbb{Q}_p$ since this is the setting of Neretin's works \cite{NeretinHuaMeasures}, \cite{NeretinBruhatTits} (in particular Proposition \ref{ConstructionInfiniteMeasure} above) and in order to simplify the exposition.
\end{rmk}

\begin{rmk} Observe that, we can identify $\Delta_0$ with the space of all partitions or equivalently all Young diagrams. Then, $\nu^{(s)}$ coincides with the probability measure on partitions $M_{(u,q)}$ from \cite{FulmanConjugacyClasses}, with the identification $u=p^{-s},q=p$, which arises in the study of the conjugacy classes of the finite general linear groups. This can be further identified, see Theorem 10 in \cite{FulmanConjugacyClasses}, with a special case of the Hall-Littlewood measures on partitions, a sub-class themselves of Macdonald measures and processes \cite{MacdonaldProcesses}.
\end{rmk}

\begin{rmk}Using the Rogers-Ramanujan identities, see for example \cite{FulmanRogersRamanujan1}, \cite{FulmanRogersRamanujan2}, it is possible to compute explicitly the cumulative distribution function of $k_1$ under $\nu^{(s)}$ for the special cases $s=0,1$. Namely, with $x\in \mathbb{N}$ with $x\ge 2$ we have:
\begin{align*}
\nu^{(0)}(k_1<x)&=\prod_{\substack{i=1\\i\equiv 0, \pm x \ \left(\textnormal{mod} \ 2x+1\right)} }^{\infty}\left(1-p^{-i}\right),\\
\nu^{(1)}(k_1<x)&=\prod_{\substack{i=2\\i\equiv 0, \pm 1 \ \left(\textnormal{mod} \ 2x+1\right)} }^{\infty}\left(1-p^{-i}\right).
\end{align*}
\end{rmk}

%\begin{rmk}\label{RemarkDeterminantal}
%By general results of Borodin \cite{BorodinMarkov} any loop-free Markov chain (namely its trajectories almost surely do not pass through the same point twice) on a discrete space $\mathfrak{X}$ gives rise to a determinantal point process by considering its random trajectories as a probability measure on $2^{\mathfrak{X}}$, the set of all subsets of $\mathfrak{X}$. Our Markov chain from Theorem \ref{MainResult} above is clearly not loop-free, but there is a simple remedy for this. One can consider its graph $\big\{\left(i,\mathsf{X}_i(\mathbf{k})\right)\big\}_{i=1}^\infty$ which is a loop-free Markov chain on the extended state space $\mathbb{Z}_{>0}\times \mathbb{Z}_+$ and thus the results of \cite{BorodinMarkov} are applicable. The correlation kernel $\mathsf{K}^{(s)}$ is then given in terms of the initial distribution $\pi^{(s)}$ and transition kernel $\mathsf{P}^{(s)}$, see Theorem 1.1 in \cite{BorodinMarkov} for the details. It would be interesting to try to compute this kernel explicitly and understand what kind of information can be extracted from it.
%\end{rmk}

\subsection{Further directions}\label{FurtherDirectionsSection}
A number of interesting questions of classical random matrix theory over the real and complex numbers also make sense over the $p$-adics. Here, we point out two directions, closely related to the subject of this paper, which we believe are promising since the investigations of the corresponding questions over the complex numbers have been fruitful, see \cite{BorodinOlshanskiErgodic}, \cite{HuaPickrellDiffusions}, \cite{MatrixBougerol}.

\paragraph{Hua measures on $\textnormal{Sym}\left(\mathbb{N},\mathbb{Q}_p\right)$.} Neretin in the same paper \cite{NeretinHuaMeasures} where he introduced $\mathsf{M}^{(s)}$, also constructed the analogue $\mathsf{\Lambda}^{(s)}$ of the Hua measures on infinite symmetric $p$-adic matrices $\textnormal{Sym}\left(\mathbb{N},\mathbb{Q}_p\right)$. Moreover, Bufetov and Qiu in \cite{BufetovQiuClassificaiton} classified the ergodic measures on $\textnormal{Sym}\left(\mathbb{N},\mathbb{Q}_p\right)$ invariant under the natural action of $\textnormal{GL}\left(\infty,\mathbb{Q}_p\right)$:
\begin{align*}
\left(\mathbf{B},\mathbf{S}\right)\mapsto\mathbf{B}\mathbf{S}\mathbf{B}^{t}, \ \ \mathbf{B} \in \textnormal{GL}(\infty,\mathbb{Z}_p),\  \mathbf{S}\in \textnormal{Sym}\left(\mathbb{N},\mathbb{Q}_p\right),
\end{align*}
where $\mathbf{B}^{t}$ is the transposition of $\mathbf{B}$. Thus, a natural question is to study the ergodic decomposition of $\mathsf{\Lambda}^{(s)}$.

Looking at the formula of Neretin \cite{NeretinHuaMeasures} for the pushforward of $\mathsf{\Lambda}^{(s)}$ on $\textnormal{Sym}\left(N,\mathbb{Q}_p\right)$ we think that it is quite plausible that there exists a Markov chain structure underlying this measure as well. Nevertheless, there is a number of additional subtleties present when it comes to studying measures on $\textnormal{Sym}\left(\mathbb{N},\mathbb{Q}_p\right)$ compared to $\textnormal{Mat}\left(\mathbb{N},\mathbb{Q}_p\right)$, see \cite{NeretinHuaMeasures}, \cite{BufetovQiuClassificaiton} for more details, so we do expect that the solution to this problem of ergodic decomposition will be more involved.

\paragraph{Stochastic dynamics.} Another natural direction is to construct, for each $N\ge 1$, Markov dynamics $\left(\mathbb{X}_N^{(s)}(t);t\ge 0\right)$ on $\textnormal{Mat}\left(N,\mathbb{Q}_p\right)$, leaving the $p$-adic Hua measures $\mathsf{M}_N^{(s)}$ invariant (or having $\mathsf{M}_N^{(s)}$ as a fixed time distribution), and which are consistent under the corners maps $\Pi_N^{N+1}$. It would then be very interesting to understand how the singular numbers evolve under such dynamics. One would also hope that there would be some connections with dynamics for Hall-Littlewood measures and processes, see for example \cite{BorodinPetrov}, \cite{BufetovPetrovLLN}, \cite{BufetovPetrovYangBaxter}. The analogous questions over the complex numbers (in the Hermitian case) were investigated in \cite{HuaPickrellDiffusions}, \cite{MatrixBougerol}.

In fact, as far as we are aware, even the simpler stochastic processes, the analogues of Dyson Brownian motion on $\textnormal{Sym}\left(N,\mathbb{Q}_p\right)$ and Ginibre Brownian motion on $\textnormal{Mat}\left(N,\mathbb{Q}_p\right)$ have not yet been introduced. For $N=1$, there is vast literature on different constructions, some of them equivalent, of diffusion and continuous time random walks on $\mathbb{Q}_p$, as a first reference see for example \cite{AlbeverioDiffusion}, \cite{AlbeverioRandomWalks}, \cite{DelMutoFigaTalamanca}, \cite{EvansLevy} (also see \cite{EvansLocalField} for a stochastic process indexed by the $p$-adics). However, it is not clear to us, at least at present, which of these stochastic processes, if any, is the correct choice to generalize to the matrix setting.

\paragraph{Acknowledgements} I would like to thank Valeriya Kovaleva for a useful discussion on background for $p$-adic matrices. I am also grateful to Leonid Petrov for a number of interesting remarks. I would like to thank the anonymous referees for a careful reading of the paper and comments and suggestions which have improved the presentation. The research described here was partly supported by ERC Advanced Grant  740900 (LogCorRM).  

\section{Asymptotic approximation for the ergodic decomposition measure}\label{AsymptoticApproximationSection}

The goal of this section is to give a concrete way to describe the abstract measure $\nu^{(s)}$ from Proposition \ref{ErgodicDecompositionIntroduction}. The approach we take is an adaptation to this setting of the Vershik-Kerov ergodic method \cite{VershikErgodic} and crucially relies on the classification results of Bufetov and Qiu \cite{BufetovQiuClassificaiton} and some general arguments of Borodin and Olshanski that can be found for example in \cite{BorodinOlshanskiErgodic}. This approach applies to any probability measure in $\mathcal{P}_{\textnormal{inv}}\left(\textnormal{Mat}\left(\mathbb{N},\mathbb{Q}_p\right)\right)$, not necessarily $\mathsf{M}^{(s)}$. We begin with the following proposition, again a consequence of the results of \cite{Bufetov} and Theorem \ref{ClassificationTheorem} (comparing with the notation conventions used for Proposition \ref{ErgodicDecompositionIntroduction} we have $\nu^{(s)}=\nu^{\mathsf{M}^{(s)}}$).
\begin{prop}\label{ErgodicDecompositionProp}
Let $\mathcal{M}\in \mathcal{P}_{\textnormal{inv}}\left(\textnormal{Mat}\left(\mathbb{N},\mathbb{Q}_p\right)\right)$. Then, there exists a unique probability measure $\nu^{\mathcal{M}}$ on $\Delta$ such that:
\begin{align*}
\mathcal{M}=\int_{\Delta}^{}\mu_{\mathbf{k}}\nu^{\mathcal{M}}(d\mathbf{k}),
\end{align*}
which means that for any bounded Borel function $\mathfrak{F}$ on $\textnormal{Mat}\left(\mathbb{N},\mathbb{Q}_p\right)$:
\begin{align*}
\mathcal{M}\left(\mathfrak{F}\right)=\int_{\Delta}^{}\mu_{\mathbf{k}}\left(\mathfrak{F}\right)\nu^{\mathcal{M}}(d\mathbf{k}).
\end{align*}
\end{prop}
We need a number of definitions and notation that will be used throughout the paper. For any $N\ge 1$ we define:
\begin{align*}
\Delta_N=\big\{\mathbf{k}=\left(k_1,\dots,k_N\right)\in \left(\mathbb{Z}\cup\{-\infty\}\right)^N:k_1\ge k_2\ge \cdots \ge k_N \big\}.
\end{align*}
For a matrix $\mathbf{Z}\in \textnormal{Mat}\left(N,\mathbb{Q}_p\right)$ with decomposition as in (\ref{SingularDecomposition}) we consider the map $\mathsf{Sing}_N:\textnormal{Mat}\left(N,\mathbb{Q}_p\right) \to \Delta_N$ given by:
\begin{align*}
\mathsf{Sing}_N\left(\mathbf{Z}\right)=\left(k_1\ge k_2 \ge \cdots \ge k_N\right).
\end{align*}
Then, for $\mathbf{Z}\in \textnormal{Mat}\left(\mathbb{N},\mathbb{Q}_p\right)$ we can define $\mathbf{k}^{(N)}\left(\mathbf{Z}\right)=\left(k_1^{(N)}\left(\mathbf{Z}\right)\ge \cdots \ge k_N^{(N)}\left(\mathbf{Z}\right)\right)\in \Delta_N$ by:
\begin{align*}
\left(k_1^{(N)}\left(\mathbf{Z}\right)\ge \cdots \ge k_N^{(N)}\left(\mathbf{Z}\right)\right)=\mathsf{Sing}_N\left(\Pi_N^{\infty}\left(\mathbf{Z}\right)\right).
\end{align*}
Moreover, for each $i\in \mathbb{Z}\cup \{-\infty \}$ we define $l^{(N)}_i\left(\mathbf{Z}\right)=\#\big\{j:k_j^{(N)}\left(\mathbf{Z}\right)=i \big\}$ and observe that $\sum_{-\infty}^{\infty}l_i^{(N)}(\mathbf{Z})=N$.
\begin{defn} We define the subset $\textnormal{Mat}_{\textnormal{reg}}\left(\mathbb{N},\mathbb{Q}_p\right)\subset \textnormal{Mat}\left(\mathbb{N},\mathbb{Q}_p\right)$ by requiring that for each $\mathbf{Z}\in \textnormal{Mat}_{\textnormal{reg}}\left(\mathbb{N},\mathbb{Q}_p\right)$ the following limits exist:
\begin{align*}
\lim_{N\ge j}k_j^{(N)}\left(\mathbf{Z}\right)\overset{\textnormal{def}}{=}k_j\left(\mathbf{Z}\right), \ j=1,2,3,\dots,
\end{align*}
in which case we write $\mathbf{k}^{(N)}\left(\mathbf{Z}\right)\longrightarrow \mathbf{k}\left(\mathbf{Z}\right)$. Furthermore, for each $i\in \mathbb{Z}\cup \{-\infty \}$ we define $l_i\left(\mathbf{Z}\right)=\#\big\{j:k_j\left(\mathbf{Z}\right)=i \big\}$.
\end{defn}
Observe that, $\textnormal{Mat}_{\textnormal{reg}}\left(\mathbb{N},\mathbb{Q}_p\right)$ is a Borel subset of $\textnormal{Mat}\left(\mathbb{N},\mathbb{Q}_p\right)\simeq \mathbb{Q}_p^{\mathbb{N}\times \mathbb{N}}$ equipped with Tychonoff's product topology. With all these definitions in place we can now state the main result of this section. This will be a consequence of the results on the classification of ergodic measures from \cite{BufetovQiuClassificaiton} and some generic abstract arguments of Borodin and Olshanski from \cite{BorodinOlshanskiErgodic}. Since the arguments for the proof of Proposition \ref{PropositionApproximation} are rather abstract and very different from the ones required in the rest of the paper, the proof is deferred to the Appendix in Section \ref{Appendix}.

\begin{prop}\label{PropositionApproximation}
Let $\mathcal{M} \in \mathcal{P}_{\textnormal{inv}}\left(\textnormal{Mat}\left(\mathbb{N},\mathbb{Q}_p\right)\right)$. Then, $\mathcal{M}$ is supported on $\textnormal{Mat}_{\textnormal{reg}}\left(\mathbb{N},\mathbb{Q}_p\right)$. Moreover, the distribution of $\mathbf{k}$ under $\nu^{\mathcal{M}}$ coincides with:
\begin{align*}
\mathsf{Law}\left(k_1\left(\mathbf{Z}\right),k_2\left(\mathbf{Z}\right),k_3\left(\mathbf{Z}\right),\dots\right), \ \textnormal{ with } \mathsf{Law}\left(\mathbf{Z}\right)=\mathcal{M}.
\end{align*}
\end{prop}

In plain words what the proposition above says is the following: in order to describe the ergodic decomposition measure $\nu^{\mathcal{M}}$ one needs to study the asymptotic singular numbers of the corners of an $\mathcal{M}$-distributed infinite random matrix $\mathbf{Z}$.

\section{Proof of the main result}\label{ProofSection}
As explained in Section \ref{AsymptoticApproximationSection} in order to prove Theorem \ref{MainResult} we need to study the $N\to \infty$ limit of the singular numbers of $\mathsf{M}_N^{(s)}$-distributed random matrices. The following proposition gives a preliminary formula for the induced law of these singular numbers that we denote by $\mathsf{m}_N^{(s)}$.
\begin{prop}\label{SingValuesMeasure1}
Let $N\in\mathbb{N}$ and $s>-1$. Then,
\begin{align*}
\mathsf{m}_N^{(s)}\left(\mathbf{k}\right)=\mathsf{m}_N^{(s)}\left(k_1,\dots,k_N\right)&=\left[\left(\mathsf{Sing}_N\right)_*\mathsf{M}_N^{(s)}\right]\left(k_1,\dots,k_N\right)\\
&=\frac{\left(p^{-1-s};p^{-1}\right)^2_{N}}{\left(p^{-1-s};p^{-1}\right)_{2N}}p^{-(s+2N)\sum_{k_j>0}^{}k_j-\sum_{j=1}^{N}(2j-2N-1)k_j}\frac{\left(p^{-1};p^{-1}\right)_N^2}{\prod_{-\infty}^{\infty}\left(p^{-1};p^{-1}\right)_{l_i\left(\mathbf{k}\right)}},
\end{align*}
where $l_{i}=l_i\left(\mathbf{k}\right)=\#\big\{j:k_j=i \big\}$.
\end{prop}
\begin{proof} 
The statement of the proposition follows immediately from formula (\ref{HuaMeasuresDefinition}) and the fact that:
\begin{align*}
\left[\left(\mathsf{Sing}_N\right)_*\mathsf{vol}\right]\left(k_1,\dots,k_N\right)=p^{-\sum_{i=1}^{N}\left(2i-2N-1\right)k_i}\frac{\left(p^{-1};p^{-1}\right)_N^2}{\prod_{-\infty}^{\infty}\left(p^{-1};p^{-1}\right)_{l_i\left(\mathbf{k}\right)}},
\end{align*}
which can be easily deduced from results in Chapter V of Macdonald's classical book on symmetric functions \cite{Macdonald}. More precisely, it is established by using display (2.9) on page 298 in Chapter V of \cite{Macdonald}, which with the notation conventions of the present paper reads as follows:
\begin{align*}
\mathsf{Haar}\left(\textnormal{GL}\left(N,\mathbb{Z}_p\right)\textnormal{diag}\left(p^{-k_1},p^{-k_2},\dots,p^{-k_N}\right) \textnormal{GL}\left(N,\mathbb{Z}_p\right)\right)=p^{-\sum_{i=1}^{N}\left(2i-N-1\right)k_i}\frac{\left(p^{-1};p^{-1}\right)_N}{\prod_{-\infty}^{\infty}\left(p^{-1};p^{-1}\right)_{l_i\left(\mathbf{k}\right)}}
\end{align*}
and the formula:
\begin{align*}
d\mathsf{vol}\left(\mathbf{Z}\right)&=\left(p^{-1};p^{-1}\right)_N|\det\left(\mathbf{Z}\right)|^{N}d\mathsf{Haar}\left(\mathbf{Z}\right)\\
&=\left(p^{-1};p^{-1}\right)_Np^{N\sum_{j=1}^{N}k_j}d\mathsf{Haar}\left(\mathbf{Z}\right).
\end{align*}
\end{proof}

Although the formula for $\mathsf{m}_N^{(s)}$ in Proposition \ref{SingValuesMeasure1} above has a rather pleasant form it is not very convenient for analysis since it involves both the variables $\{k_i\}_{i=1}^N$ and $\{l_i(\mathbf{k}) \}_{-\infty}^\infty$. Our next goal is to write this formula only in terms of the $\{l_i(\mathbf{k}) \}_{-\infty}^\infty$. This is achieved through the following elementary combinatorial lemma.

\begin{lem}\label{LemmaRewriting} Let $\mathbf{k}=\left(k_1,k_2,\dots,k_N\right)\in \Delta_N$; so that in particular $\sum_{-\infty}^{\infty}l_j\left(\mathbf{k}\right)=N$. Then, we have:
\begin{enumerate}
\item $\sum_{j=1}^{N}\mathbf{1}\left(k_j>0\right)k_j\left(2j-1\right)=\left(\sum_{j=1}^{\infty}l_j\left(\mathbf{k}\right)\right)^2+\left(\sum_{j=2}^{\infty}l_j\left(\mathbf{k}\right)\right)^2+\left(\sum_{j=3}^{\infty}l_j\left(\mathbf{k}\right)\right)^2+\cdots\\=\left(N-\sum_{-\infty}^{0}l_j\left(\mathbf{k}\right)\right)^2+\left(N-\sum_{-\infty}^{1}l_j\left(\mathbf{k}\right)\right)^2+\left(N-\sum_{-\infty}^{2}l_j\left(\mathbf{k}\right)\right)^2+\cdots$.
\item $\sum_{j=1}^{N}\mathbf{1}\left(k_j\le 0\right)k_j\left(2j-2N-1\right)=\left(\sum_{-\infty}^{-1}l_j\left(\mathbf{k}\right)\right)^2+\left(\sum_{-\infty}^{-2}l_j\left(\mathbf{k}\right)\right)^2+\left(\sum_{-\infty}^{-3}l_j\left(\mathbf{k}\right)\right)^2+\cdots$.
\item $\sum_{j=1}^{N}\mathbf{1}\left(k_j>0\right)k_j=\sum_{j=0}^{\infty}jl_j\left(\mathbf{k}\right)=\sum_{j=1}^{\infty}jl_j\left(\mathbf{k}\right).$
\end{enumerate}
\end{lem}

\begin{proof}
Item 3 is essentially immediate by definition. We then only prove item 1 since item 2 follows by analogous arguments. Item 1 is equivalent to proving the following for a partition $\lambda=\left(1^{m_1}2^{m_2}\dots r^{m_r}\dots\right)$:
\begin{align*}
\sum_{i=1}^{\infty}(2i-1)\lambda_i=\sum_{i=1}^{\infty}\left(\sum_{j=i}^{\infty}m_i(\lambda)\right)^2,
\end{align*}
where by definition $m_i(\lambda)$ is the number of parts in $\lambda$ equal to $i$. On the other hand, if we consider the conjugate partition $\lambda'$, using the relation $m_i(\lambda)=\lambda_i'-\lambda_{i+1}'$ we are required to show that:
\begin{align*}
\sum_{i=1}^{\infty}(2i-1)\lambda_i=\sum_{i=1}^{\infty}\left(\lambda_i'\right)^2.
\end{align*}
Using the formulae (1.5) and (1.6) on page 3 of \cite{Macdonald} we can easily establish this as follows:
\begin{align*}
\sum_{i=1}^{\infty}(2i-1)\lambda_i=\sum_{i=1}^{\infty}\left[\left(\lambda_i'\right)^2-\lambda_i'\right]+\sum_{i=1}^{\infty}\lambda_i=\sum_{i=1}^{\infty}\left[\left(\lambda_i'\right)^2-\lambda_i'\right]+\sum_{i=1}^{\infty}\lambda_i'=\sum_{i=1}^{\infty}\left(\lambda_i'\right)^2.
\end{align*}
\end{proof}

The following formula for the law $\mathsf{m}_N^{(s)}$ of the singular numbers will be the starting point of our analysis.

\begin{prop}\label{SingValuesMeasure2}
Let $N\in\mathbb{N}$ and $s>-1$. Then,
\begin{align*}
&\mathsf{m}_N^{(s)}\left(k_1,\dots,k_N\right)=\mathsf{m}_N^{(s)}\left(\{l_i\left(\mathbf{k}\right) \}_{-\infty}^{\infty}\right)\\
&=\frac{\left(p^{-1-s};p^{-1}\right)^2_{N}\left(p^{-1};p^{-1}\right)_N^2}{\left(p^{-1-s};p^{-1}\right)_{2N}}p^{-\left[s\sum_{j=0}^{\infty}jl_j\left(\mathbf{k}\right)+\sum_{i=1}^{\infty}\left(\sum_{j=i}^{\infty}l_j\left(\mathbf{k}\right)\right)^2+\sum_{i=1}^{\infty}\left(\sum_{j=-\infty}^{-i}l_j\left(\mathbf{k}\right)\right)^2\right]}\frac{1}{\prod_{-\infty}^{\infty}\left(p^{-1};p^{-1}\right)_{l_i\left(\mathbf{k}\right)}},
\end{align*}
where $l_{i}=l_i\left(\mathbf{k}\right)=\#\big\{j:k_j=i \big\}$.
\end{prop}

\begin{proof}
Immediate rewriting of Proposition \ref{SingValuesMeasure1} by means of Lemma \ref{LemmaRewriting}.
\end{proof}

We now arrive at the key ingredient for establishing Theorem \ref{MainResult}, which is a Markov chain representation of the law $\mathsf{m}_N^{(s)}$ of the singular numbers. The proof is a simple elementary verification and thus not particularly illuminating. What is really non-trivial here is discovering the statement of the result. We arrived at it by specifically looking for such a Markov chain structure, partly motivated by the simpler case\footnote{\label{Evansfootnote}Evans in \cite{Evans} studies the probability measure $d\mathsf{vol}$ on $\textnormal{Mat}\left(N,\mathbb{Z}_p\right)$, in particular in this case all the singular numbers $(k_1,\dots,k_N)$ are non-positive. Note that, this probability measure coincides with $\mathsf{Law}\left(\Pi_N^{\infty}\left(\mathbf{A}_{\mathbf{0}}\right)\right)=\left(\Pi_N^{\infty}\right)_*\mu_{\mathbf{0}}$, where $\mathbf{0}=(0,0,0,\dots)\in \Delta$ and $\mathbf{A}_{\mathbf{k}}$ and $\mu_{\mathbf{k}}$ were introduced in Definition \ref{DefinitionErgodicMeasure}. From Proposition \ref{PropositionApproximation} we readily get that, since $\mu_{\mathbf{0}}$ is an ergodic measure, $\lim_{N\ge j}k_j^{(N)}\left(\mathbf{A}_{\mathbf{0}}\right)=0$, for all $j\ge 1$ and we observe that these limits can also be established using the Markov chain representation obtained by Evans in \cite{Evans}. Finally, it is worth mentioning that the proof by Evans in \cite{Evans} of an underlying Markov chain structure in $\left(\mathsf{Sing}_N\Pi_N^{\infty}\right)_*\mu_{\mathbf{0}}$ uses in a very essential way the fact that the entries of $\mathbf{A}_{\mathbf{0}}$ are i.i.d. and thus we do not see how it can be adapted to our setting.} investigated by Evans in \cite{Evans}, but mainly since the tools previously used for this class of problems \cite{BorodinOlshanskiErgodic}, \cite{BorodinOlshanskiHarmonic}, \cite{GorinOlshanski}, \cite{Cuenca}, \cite{AssiotisInverseWishart}, such as determinantal point processes, do not appear to be applicable here. Although a conceptual explanation of why such a Markov chain structure exists behind this problem is currently lacking, it would be very interesting to have one and we are investigating it.

Recall that for $\mathbf{k}\in \Delta_N$ we have $\sum_{-\infty}^{\infty}l_i(\mathbf{k})=N$ and moreover the collection of numbers $\big\{l_i(\mathbf{k})\big\}_{-\infty}^\infty$ is uniquely determined by $\bigg\{\sum_{j=i}^{\infty}l_j(\mathbf{k})\bigg\}_{i=0}^\infty\cup\bigg\{N-\sum_{j=-i}^{\infty}l_j(\mathbf{k})\bigg\}_{i=1}^\infty$ and vice versa.

\begin{thm}\label{MarkovChainRep1}
Let $N\in\mathbb{N}$ and $s>-1$. Define $\tilde{\pi}_N^{(s)}$ by the formula:
\begin{align*}
\tilde{\pi}_N^{(s)}(x)=\frac{\left(p^{-1-s};p^{-1}\right)_N^2\left(p^{-1};p^{-1}\right)_N^2}{\left(p^{-1-s};p^{-1}\right)_{2N}\left(p^{-1};p^{-1}\right)_x\left(p^{-1-s};p^{-1}\right)_x\left(p^{-1};p^{-1}\right)^2_{N-x}}p^{-(N-x)^2}, \ 0\le x\le N.
\end{align*}
Then, $\tilde{\pi}_N^{(s)}$ is a probability measure on $\llbracket0,N \rrbracket$. Moreover, we can sample a $\mathbf{k}\in \Delta_N$ having $\mathsf{Law}(\mathbf{k})=\mathsf{m}_N^{(s)}$, or equivalently the corresponding collection of numbers $\{l_i(\mathbf{k}) \}_{-\infty}^\infty$, as follows, see Figure \ref{Figure2} for an illustration:
\begin{enumerate}
\item First, sample $\sum_{0}^{\infty}l_i(\mathbf{k})$ according to $\tilde{\pi}_N^{(s)}$.
\item Starting from $\sum_{0}^{\infty}l_i(\mathbf{k})$ we sample $\sum_{1}^{\infty}l_i(\mathbf{k}),\sum_{2}^{\infty}l_i(\mathbf{k}), \sum_{3}^{\infty}l_i(\mathbf{k}),\dots$ by running a Markov chain with transition kernel $\mathsf{P}^{(s)}$. We note that the Markov kernel $\mathsf{P}^{(s)}$, which was defined in (\ref{DefMarkovKernel}), is the same kernel appearing in Theorem \ref{MainResult} and in particular is independent of $N$.
\item Starting from $N-\sum_{0}^{\infty}l_i(\mathbf{k})$ we sample $N-\sum_{-1}^{\infty}l_i(\mathbf{k}),N-\sum_{-2}^{\infty}l_i(\mathbf{k}), N-\sum_{-3}^{\infty}l_i(\mathbf{k}),\dots$ by running a Markov chain with transition kernel $\mathsf{P}=\mathsf{P}^{(0)}$.
\end{enumerate}
\end{thm}

\begin{proof}
Let $\mathbf{k}\in \Delta_N$ be arbitrary. We first compute:
\begin{align*}
\prod_{i=0}^{\infty}\mathsf{P}^{(s)}\left(\sum_{j=i}^{\infty}l_j\left(\mathbf{k}\right),\sum_{j=i+1}^{\infty}l_j\left(\mathbf{k}\right)\right)&=p^{-\sum_{i=1}^{\infty}\left(\sum_{j=i}^{\infty}l_j(\mathbf{k})\right)^2-s\sum_{j=1}^{\infty}jl_j(\mathbf{k})}\times\\&\times\left(p^{-1};p^{-1}\right)_{\sum_{j=0}^{\infty}l_j(\mathbf{k})}\left(p^{-1-s};p^{-1}\right)_{\sum_{j=0}^{\infty}l_j(\mathbf{k})}\prod_{i=1}^{\infty}\frac{1}{\left(p^{-1};p^{-1}\right)_{l_i(\mathbf{k})}},\\
\prod_{i=0}^{\infty}\mathsf{P}\left(N-\sum_{j=-i}^{\infty}l_j\left(\mathbf{k}\right),N-\sum_{j=-i-1}^{\infty}l_j\left(\mathbf{k}\right)\right)&=p^{-\sum_{i=1}^{\infty}\left(	N-\sum_{j=-i}^{\infty}l_j(\mathbf{k})\right)^2}\left(p^{-1};p^{-1}\right)_{N-\sum_{j=0}^{\infty}l_j(\mathbf{k})}\times\\&\times\prod_{i=0}^{\infty}\frac{1}{\left(p^{-1};p^{-1}\right)_{l_{-i}(\mathbf{k})}}\\
&=p^{-\sum_{i=1}^{\infty}\left(	\sum_{-\infty}^{-i-1}l_j(\mathbf{k})\right)^2}\left(p^{-1};p^{-1}\right)_{N-\sum_{j=0}^{\infty}l_j(\mathbf{k})}\prod_{i=0}^{\infty}\frac{1}{\left(p^{-1};p^{-1}\right)_{l_{-i}(\mathbf{k})}}.
\end{align*}
Thus, we have:
\begin{align}\label{MarkovRepDisplay1}
\tilde{\pi}_N^{(s)}\left(\sum_{i=0}^{\infty}l_i(\mathbf{k})\right)\prod_{i=0}^{\infty}\mathsf{P}^{(s)}\left(\sum_{j=i}^{\infty}l_j\left(\mathbf{k}\right),\sum_{j=i+1}^{\infty}l_j\left(\mathbf{k}\right)\right)\prod_{i=0}^{\infty}\mathsf{P}\left(N-\sum_{j=-i}^{\infty}l_j\left(\mathbf{k}\right),N-\sum_{j=-i-1}^{\infty}l_j\left(\mathbf{k}\right)\right)\nonumber\\
=\mathsf{m}_N^{(s)}\left(\bigg\{\sum_{j=i}^{\infty}l_j(\mathbf{k})\bigg\}_{i=0}^\infty\cup\bigg\{N-\sum_{j=-i}^{\infty}l_j(\mathbf{k})\bigg\}_{i=1}^\infty\right)=\mathsf{m}_N^{(s)}\left(\big\{l_i(\mathbf{k})\big\}_{-\infty}^\infty\right).
\end{align}
Since $\mathsf{m}_N^{(s)}$ is a probability measure, $\mathsf{P}^{(s)}$ is a Markov kernel and $\mathbf{k}\in \Delta_N$ was arbitrary we obtain that $\sum_{x=0}^N\tilde{\pi}_N^{(s)}(x)=1$, which gives (since it is clearly positive) that $\tilde{\pi}_N^{(s)}$ is a probability measure on $\llbracket0,N \rrbracket$. Thus, (\ref{MarkovRepDisplay1}) gives us the statement of the theorem.
\end{proof}

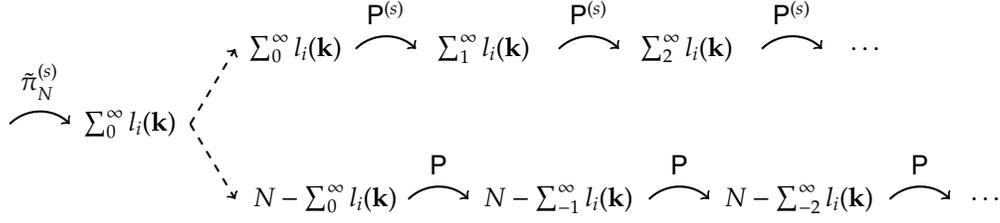
\begin{figure}
\captionsetup{singlelinecheck = false, justification=justified}
\begin{tikzpicture}
      \draw[->,thick] (0,0) to [out=45,in=135] (0.8,0);
               
              \node[above] at (0.4,0.2) {$\tilde{\pi}_N^{(s)}$};
              \node[] at (1.6,0) {$\sum_{0}^{\infty}l_i(\mathbf{k})$};
              
               \draw[->,dashed,thick] (2.4,0) to (3,1);
                             \draw[->,dashed,thick] (2.4,0) to (3,-1);

              \node[] at (3.8,1) {$\sum_{0}^{\infty}l_i(\mathbf{k})$};
              \node[] at (4.2,-1) {$N-\sum_{0}^{\infty}l_i(\mathbf{k})$};
              
              \draw[->,thick] (4.6,1) to [out=45,in=135] (5.4,1);
             \node[above] at (5,1.2) {$\mathsf{P}^{(s)}$};   
             
                           \draw[->,thick] (5.3,-1) to [out=45,in=135] (6.1,-1);    
                          \node[above] at (5.7,-0.8) {$\mathsf{P}$};           
                          
            \node[] at (7.3,-1) {$N-\sum_{-1}^{\infty}l_i(\mathbf{k})$};                                    
                                                  
              \node[] at (6.3,1) {$\sum_{1}^{\infty}l_i(\mathbf{k})$};
              
                            \draw[->,thick] (7.3,1) to [out=45,in=135] (8.1,1);
                                         \node[above] at (7.7,1.2) {$\mathsf{P}^{(s)}$};             
                                         
                                \draw[->,thick] (8.5,-1) to [out=45,in=135] (9.3,-1);    
                                     \node[above] at (8.9,-0.8) {$\mathsf{P}$};

             \node[] at (9,1) {$\sum_{2}^{\infty}l_i(\mathbf{k})$};
               \node[] at (10.5,-1) {$N-\sum_{-2}^{\infty}l_i(\mathbf{k})$};

    \draw[->,thick] (11.7,-1) to [out=45,in=135] (12.5,-1);
    \node[above] at (12.1,-0.8) {$\mathsf{P}$};                                  
    
                                         \draw[->,thick] (10,1) to [out=45,in=135] (10.8,1);
                                        \node[above] at (10.4,1.2) {$\mathsf{P}^{(s)}$};                                            
              \node[] at (11.4,1) {$\dots$};
                            \node[] at (13,-1) {$\dots$};
 \end{tikzpicture}
 \caption{An illustration of the Markov chain from Theorem \ref{MarkovChainRep1}. Here, the dashed arrows represent a deterministic move while the solid ones a random move (according to the corresponding probability measure).}\label{Figure2}
 \end{figure}
We also have the following alternative Markov chain representation of $\mathsf{m}_N^{(s)}$ which will be used in the sequel as well. As before, observe that for $\mathbf{k}\in \Delta_N$ the collection of numbers $\big\{l_i(\mathbf{k})\big\}_{-\infty}^\infty$ is uniquely determined by $\bigg\{N-\sum_{-\infty}^{i}l_j(\mathbf{k})\bigg\}_{i=1}^\infty\cup \bigg\{\sum_{-\infty}^{-i}l_j(\mathbf{k})\bigg\}_{i=0}^\infty$ and vice versa.

\begin{thm}\label{MarkovChainRep2}
Let $N\in\mathbb{N}$ and $s>-1$. Define $\pi_N^{(s)}$ by the formula:
\begin{align*}
\pi_N^{(s)}(x)=\frac{\left(p^{-1-s};p^{-1}\right)_N^2\left(p^{-1};p^{-1}\right)_N^2}{\left(p^{-1-s};p^{-1}\right)_{2N}\left(p^{-1};p^{-1}\right)_x^2\left(p^{-1};p^{-1}\right)_{N-x}\left(p^{-1-s};p^{-1}\right)_{N-x}}p^{-(N-x)^2-s(N-x)}, \ 0\le x\le N.
\end{align*}
Then, $\pi_N^{(s)}$ is a probability measure on $\llbracket0,N \rrbracket$. Moreover, we can sample a $\mathbf{k}\in \Delta_N$ having $\mathsf{Law}(\mathbf{k})=\mathsf{m}_N^{(s)}$, or equivalently the corresponding collection of numbers $\{l_i(\mathbf{k}) \}_{-\infty}^\infty$, as follows, see Figure \ref{Figure3} for an illustration:
\begin{enumerate}
\item First, sample $\sum_{-\infty}^{0}l_i(\mathbf{k})$ according to $\pi_N^{(s)}$.
\item Starting from $N-\sum_{-\infty}^{0}l_i(\mathbf{k})$ we sample $N-\sum_{-\infty}^{1}l_i(\mathbf{k}),N-\sum_{-\infty}^{2}l_i(\mathbf{k}), N-\sum_{-\infty}^{3}l_i(\mathbf{k}),\dots$ by running a Markov chain with transition kernel $\mathsf{P}^{(s)}$.
\item Starting from $\sum_{-\infty}^{0}l_i(\mathbf{k})$ we sample $\sum_{-\infty}^{-1}l_i(\mathbf{k}), \sum_{-\infty}^{-2}l_i(\mathbf{k}),  \sum_{-\infty}^{-3}l_i(\mathbf{k}),\dots$ by running a Markov chain with transition kernel $\mathsf{P}=\mathsf{P}^{(0)}$.
\end{enumerate}
\end{thm}

\begin{proof}
The verification is analogous to the proof of Theorem \ref{MarkovChainRep1}. Let $\mathbf{k}\in \Delta_N$ be arbitrary and compute:
\begin{align*}
\prod_{i=0}^{\infty}\mathsf{P}^{(s)}\left(N-\sum_{-\infty}^{i}l_j(\mathbf{k}),N-\sum_{-\infty}^{i+1}l_j(\mathbf{k})\right)&=p^{-\sum_{i=2}^{\infty}\left(\sum_{j=i}^{\infty}l_j(\mathbf{k})\right)^2-s\sum_{j=1}^{\infty}(j-1)l_j(\mathbf{k})}\left(p^{-1};p^{-1}\right)_{N-\sum_{-\infty}^{0}l_j(\mathbf{k})}\times\\ &\times\left(p^{-1-s};p^{-1}\right)_{N-\sum_{-\infty}^{0}l_j(\mathbf{k})} \prod_{i=1}^{\infty}\frac{1}{\left(p^{-1};p^{-1}\right)_{l_{i}(\mathbf{k})}},  \\
\prod_{i=0}^{\infty}\mathsf{P}\left(\sum_{-\infty}^{-i}l_j(\mathbf{k}),\sum_{-\infty}^{-i-1}l_j(\mathbf{k})\right)&=p^{-\sum_{i=1}^{\infty}\left(\sum_{-\infty}^{-i}l_j(\mathbf{k})\right)^2}\left(p^{-1};p^{-1}\right)_{\sum_{-\infty}^{0}l_j(\mathbf{k})}^2\prod_{i=0}^{\infty}\frac{1}{\left(p^{-1};p^{-1}\right)_{l_{-i}(\mathbf{k})}}.
\end{align*}
Thus, we have:
\begin{align}\label{MarkovRepDisplay2}
\pi_N^{(s)}\left(\sum_{-\infty}^{0}l_i(\mathbf{k})\right)\prod_{i=0}^{\infty}\mathsf{P}^{(s)}\left(N-\sum_{-\infty}^{i}l_j(\mathbf{k}),N-\sum_{-\infty}^{i+1}l_j(\mathbf{k})\right)\prod_{i=0}^{\infty}\mathsf{P}\left(\sum_{-\infty}^{-i}l_j(\mathbf{k}),\sum_{-\infty}^{-i-1}l_j(\mathbf{k})\right)\nonumber\\
=\mathsf{m}_N^{(s)}\left(\bigg\{N-\sum_{-\infty}^{i}l_j(\mathbf{k})\bigg\}_{i=1}^\infty\cup \bigg\{\sum_{-\infty}^{-i}l_j(\mathbf{k})\bigg\}_{i=0}^\infty\right)=\mathsf{m}_N^{(s)}\left(\big\{l_i(\mathbf{k})\big\}_{-\infty}^\infty\right).
\end{align}
Since $\mathsf{m}_N^{(s)}$ is a probability measure, $\mathsf{P}^{(s)}$ is a Markov kernel and $\mathbf{k}\in \Delta_N$ was arbitrary we obtain that $\pi_N^{(s)}$ is a probability measure on $\llbracket0,N \rrbracket$. Thus, (\ref{MarkovRepDisplay2}) gives us the statement of the theorem.
\end{proof}

\begin{figure}
\captionsetup{singlelinecheck = false, justification=justified}
\begin{tikzpicture}
\draw[->,thick] (0,0) to [out=45,in=135] (0.8,0);
               
              \node[above] at (0.4,0.2) {$\pi_N^{(s)}$};
              \node[] at (1.6,0) {$\sum_{-\infty}^{0}l_i(\mathbf{k})$};
              
               \draw[->,dashed,thick] (2.4,0) to (3,1);
                             \draw[->,dashed,thick] (2.4,0) to (3,-1);

              \node[] at (3.8,-1) {$\sum_{-\infty}^{0}l_i(\mathbf{k})$};
              \node[] at (4.2,1) {$N-\sum_{-\infty}^{0}l_i(\mathbf{k})$};
              
              \draw[->,thick] (4.6,-1) to [out=45,in=135] (5.4,-1);
             \node[above] at (5,-0.8) {$\mathsf{P}$};   
             
                           \draw[->,thick] (5.3,1) to [out=45,in=135] (6.1,1);    
                          \node[above] at (5.7,1.2) {$\mathsf{P}^{(s)}$};           
                          
            \node[] at (7.3,1) {$N-\sum_{-\infty}^{1}l_i(\mathbf{k})$};                                    
                                                  
              \node[] at (6.3,-1) {$\sum_{-\infty}^{-1}l_i(\mathbf{k})$};
              
                            \draw[->,thick] (7.3,-1) to [out=45,in=135] (8.1,-1);
                                         \node[above] at (7.7,-0.8) {$\mathsf{P}$};             
                                         
                                \draw[->,thick] (8.5,1) to [out=45,in=135] (9.3,1);    
                                     \node[above] at (8.9,1.2) {$\mathsf{P}^{(s)}$};

             \node[] at (9,-1) {$\sum_{-\infty}^{-2}l_i(\mathbf{k})$};
               \node[] at (10.5,1) {$N-\sum_{-\infty}^{2}l_i(\mathbf{k})$};

    \draw[->,thick] (11.7,1) to [out=45,in=135] (12.5,1);
    \node[above] at (12.1,1.2) {$\mathsf{P}^{{(s)}}$};                                  
    
                                         \draw[->,thick] (10,-1) to [out=45,in=135] (10.8,-1);
                                        \node[above] at (10.4,-0.8) {$\mathsf{P}$};                                            
              \node[] at (11.4,-1) {$\dots$};
                            \node[] at (13,1) {$\dots$};
 \end{tikzpicture}
 \caption{An illustration of the Markov chain from Theorem \ref{MarkovChainRep2}. Here, the dashed arrows represent a deterministic move while the solid ones a random move (according to the corresponding probability measure).}\label{Figure3}
 \end{figure}
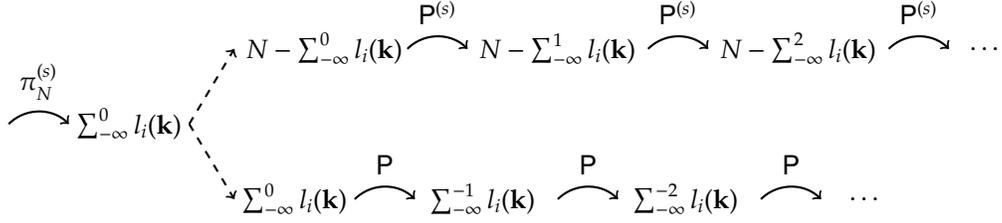

\begin{rmk}
Both identities $\sum_{x=0}^N\tilde{\pi}_N^{(s)}(x)=1$ and $\sum_{x=0}^N\pi_N^{(s)}(x)=1$ can in fact be obtained from the $q$-analogue of the Chu-Vandermonde summation formula. More precisely, using the identity, see display (1.11.5) in \cite{HypergeometricOrthogonalPolynomials}:
\begin{align*}
{}_2 \phi_1\left(\begin{matrix}q^{-N}, b\\c\end{matrix};q,q\right)
=\frac{\left(b^{-1}c;q\right)_N}{\left(c;q\right)_N}b^N,
\end{align*}
we obtain $\sum_{x=0}^N\tilde{\pi}_N^{(s)}(x)=1$ by taking $q=p^{-1}, b=q^{-N}, c=q^{1+s}$, while to get $\sum_{x=0}^N\pi_N^{(s)}(x)=1$ we take $q=p^{-1}, b=q^{-N-s}, c=q$. Nevertheless, we expect that the non-computational proof we presented earlier is new.
\end{rmk}

The following lemma is an easy consequence of Theorem \ref{MarkovChainRep2}.
\begin{lem}\label{PosNoConvLem}
Let $s>-1$. Suppose $\mathbf{Z}\in \textnormal{Mat}_{\textnormal{reg}}\left(\mathbb{N},\mathbb{Q}_p\right)$ with $\mathsf{Law}\left(\mathbf{Z}\right)=\mathsf{M}^{(s)}$. Then,
\begin{align*}
\sum_{i=1}^{\infty}l_i^{(N)}\left(\mathbf{Z}\right)=N-\sum_{-\infty}^{0}l_i^{(N)}\left(\mathbf{Z}\right)\overset{\textnormal{d}}{\longrightarrow} \mathcal{Z}, \ \textnormal{ with } \mathsf{Law}\left(\mathcal{Z}\right)=\pi^{(s)},
\end{align*}
where $\pi^{(s)}$ was given in Definition \ref{MainDefinition}.
\end{lem}

\begin{proof}
From Theorem \ref{MarkovChainRep2} we have:
\begin{align*}
\mathsf{Law}\left(\sum_{-\infty}^{0}l_i^{(N)}\left(\mathbf{Z}\right)\right)=\pi^{(s)}_N, \ \textnormal{ where } \mathsf{Law}\left(\mathbf{Z}\right)=\mathsf{M}^{(s)},
\end{align*}
since $\mathbf{k}^{(N)}(\mathbf{Z})$ is distributed according to $\mathsf{m}_N^{(s)}$. Thus, we obtain:
\begin{align*}
\mathsf{Law}\left[\sum_{i=1}^{\infty}l_i^{(N)}\left(\mathbf{Z}\right)\right](y)&=\mathsf{Law}\left[N-\sum_{-\infty}^{0}l_i^{(N)}\left(\mathbf{Z}\right)\right](y)\\&=\frac{\left(p^{-1-s};p^{-1}\right)_N^2\left(p^{-1};p^{-1}\right)_N^2p^{-y^2-sy}}{\left(p^{-1-s};p^{-1}\right)_{2N}\left(p^{-1};p^{-1}\right)_{N-y}^2\left(p^{-1};p^{-1}\right)_y\left(p^{-1-s};p^{-1}\right)_y}.
\end{align*}
The result then follows immediately by taking the $N\to \infty$ limit in the formula above.
\end{proof}

We can now prove part of the statement of Theorem \ref{MainResult}.

\begin{prop}\label{DeltaNaughtProp}
Let $s>-1$. Then, the probability measure $\nu^{(s)}$ is supported on $\Delta_0$.
\end{prop}

\begin{proof}
We define $\Delta_{\ge 0}=\Delta\cap \mathbb{Z}_+^{\mathbb{N}}$. We will first show the weaker statement that $\nu^{(s)}$ is supported on $\Delta_{\ge 0}$. From Proposition \ref{PropositionApproximation} we need to show that for each $j\ge 1$:
\begin{align*}
k_j\left(\mathbf{Z}\right)\ge 0, \ \ \textnormal{for } \mathsf{M}^{(s)}-\textnormal{a.e.} \  \mathbf{Z}\in \textnormal{Mat}_{\textnormal{reg}}\left(\mathbb{N},\mathbb{Q}_p\right).
\end{align*}
We prove the statement above for a fixed $j\ge 1$ (the general statement clearly follows from taking a countable union). Since $\lim_{N\ge j}k_j^{(N)}\left(\mathbf{Z}\right)=k_j\left(\mathbf{Z}\right)$, it then suffices to show that:
\begin{align*}
\mathsf{M}^{(s)}\left(k_j^{(N)}\left(\mathbf{Z}\right)\ge 0 \ \textnormal{eventually}\right)=1.
\end{align*}
For any $N\ge 1$ we define the event $\mathsf{E}_N^{(j)}=\big\{k_j^{(N)}\left(\mathbf{Z}\right)<0 \big\}$. We are then required to show:
\begin{align*}
\mathsf{M}^{(s)}\left(\limsup\mathsf{E}_N^{(j)}\right)=\mathsf{M}^{(s)}\left(\mathsf{E}_N^{(j)} \textnormal{'s occur infinitely often}\right)=0.
\end{align*}
By the Borel-Cantelli lemma it suffices to show that:
\begin{align*}
\sum_{N=1}^{\infty}\mathsf{M}^{(s)}\left(\mathsf{E}_N^{(j)}\right)<\infty.
\end{align*}
Since $k_1^{(N)}\left(\mathbf{Z}\right)\ge k_2^{(N)}\left(\mathbf{Z}\right) \ge \cdots \ge k_N^{(N)}\left(\mathbf{Z}\right)$ and $\sum_{i=0}^{\infty}l_i^{(N)}\left(\mathbf{Z}\right)=\#\big\{i:k_i^{(N)}\left(\mathbf{Z}\right)\ge 0 \big\}$ we have that:
\begin{align*}
\mathsf{M}^{(s)}\left(\mathsf{E}_N^{(j)}\right)=\mathsf{M}^{(s)}\left(k_j^{(N)}\left(\mathbf{Z}\right)<0\right)=\mathsf{M}^{(s)}\left(\sum_{i=0}^{\infty}l_i^{(N)}\left(\mathbf{Z}\right)<j\right).
\end{align*}
From Theorem \ref{MarkovChainRep1}, since $\mathbf{k}^{(N)}(\mathbf{Z})$ is distributed according to $\mathsf{m}_N^{(s)}$, we know that:
\begin{align*}
\mathsf{M}^{(s)}\left(\sum_{i=0}^{\infty}l_i^{(N)}\left(\mathbf{Z}\right)<j\right)=\sum_{x=0}^{j-1}\tilde{\pi}^{(s)}_N(x).
\end{align*}
Moreover, we have the following simple estimate uniformly in $N$:
\begin{align*}
\tilde{\pi}^{(s)}_N(x)=\frac{\left(p^{-1-s};p^{-1}\right)_N^2\left(p^{-1};p^{-1}\right)_N^2p^{-\left(N-x\right)^2}}{\left(p^{-1};p^{-1}\right)_{2N}\left(p^{-1};p^{-1}\right)_x\left(p^{-1-s};p^{-1}\right)_x\left(p^{-1};p^{-1}\right)_{N-x}^2}\le C^{(s)}_j p^{-\left(N-j\right)^2}, \ \forall \  0\le x \le j,
\end{align*}
for some absolute constant $C_j^{(s)}$ (independent of $N$ and whose exact numerical value is not important here). Thus,
\begin{align*}
\mathsf{M}^{(s)}\left(\sum_{i=0}^{\infty}l_i^{(N)}\left(\mathbf{Z}\right)<j\right)\lesssim p^{-\left(N-j\right)^2}
\end{align*}
and so:
\begin{align*}
\sum_{N=1}^{\infty}\mathsf{M}^{(s)}\left(\mathsf{E}_N^{(j)}\right)=\sum_{N=1}^{\infty}\mathsf{M}^{(s)}\left(\sum_{i=0}^{\infty}l_i^{(N)}\left(\mathbf{Z}\right)<j\right)\lesssim \sum_{N=1}^{\infty} p^{-\left(N-j\right)^2}<\infty,
\end{align*}
which proves that $\nu^{(s)}$ is supported on $\Delta_{\ge 0}$. We now show that it is actually supported on $\Delta_0\subset \Delta_{\ge 0}$. Assume that it is not. Thus, since we have already shown that $\nu^{(s)}$ is supported on $\Delta_{\ge 0}$ the following should hold:
\begin{align*}
\mathsf{M}^{(s)}\left(\#\big\{i:k_i\left(\mathbf{Z}\right)>0\big\}=\infty\right)>0.
\end{align*}
Then, we must have:
\begin{align*}
\mathsf{M}^{(s)}\left(\#\big\{i:k^{(N)}_i\left(\mathbf{Z}\right)>0\big\}\overset{N\to \infty}{\longrightarrow}\infty\right)>0.
\end{align*}
On the other hand, by definition we have that:
\begin{align*}
\#\big\{i:k^{(N)}_i\left(\mathbf{Z}\right)>0\big\}=\sum_{i=1}^{\infty}l_i^{(N)}\left(\mathbf{Z}\right)=N-\sum_{-\infty}^{0}l_i^{(N)}\left(\mathbf{Z}\right).
\end{align*}
Since by Lemma \ref{PosNoConvLem} we know that
\begin{align*}
N-\sum_{-\infty}^{0}l_i^{(N)}\left(\mathbf{Z}\right)\overset{\textnormal{d}}{\longrightarrow} \mathcal{Z}, \ \textnormal{ with } \mathsf{Law}\left(\mathcal{Z}\right)=\pi^{(s)}
\end{align*}
and thus the sequence of random variables $\big\{\sum_{i=1}^\infty l_i^{(N)}\left(\mathbf{Z}\right)\big\}_{N=1}^\infty$ is tight, we obtain:
\begin{align*}
\mathsf{M}^{(s)}\left(\sum_{i=1}^\infty l_i^{(N)}\left(\mathbf{Z}\right)\overset{N\to \infty}{\longrightarrow}\infty\right)=0.
\end{align*}
This gives a contradiction and completes the proof.
\end{proof}

In order to complete the proof of Theorem \ref{MainResult} we first need a definition (that will also be used in the appendix in Section \ref{Appendix}) and a simple lemma.

\begin{defn}\label{DefinitionSeqConv}
Suppose that we are given a sequence $\{\mathbf{k}^{(N)}\}_{N\ge 1}$ so that $\mathbf{k}^{(N)}\in \Delta_N$ and $\mathbf{k}\in \Delta$. We write
\begin{align*}
\mathbf{k}^{(N)} \longrightarrow \mathbf{k}
\end{align*}
if the following limits exist:
\begin{align*}
\lim_{N\ge j}k_j^{(N)}=k_j, \ j=1,2,3,\dots .
\end{align*}
\end{defn}

\begin{lem}\label{DeterministicConvergenceLemma}
Suppose that we are given a sequence $\{\mathbf{k}^{(N)}\}_{N\ge 1}$ so that:
\begin{align}\label{LemAssumption}
\Delta_N\ni\mathbf{k}^{(N)}\longrightarrow \mathbf{k}\in\Delta_0.
\end{align}
Then, we have:
\begin{align}
\big\{ l_i\left(\mathbf{k}^{(N)}\right)\big\}_{i=1}^{\infty}&\overset{N\to \infty}{\longrightarrow}\big\{ l_i\left(\mathbf{k}\right)\big\}_{i=1}^\infty,\label{claim1}\\
\bigg\{ \sum_{j=i}^{\infty}l_j\left(\mathbf{k}^{(N)}\right)\bigg\}_{i=1}^{\infty}&\overset{N\to \infty}{\longrightarrow}\bigg\{\sum_{j=i}^\infty l_j\left(\mathbf{k}\right)\bigg\}_{i=1}^\infty\label{claim2}.
\end{align}
\end{lem}

\begin{proof}
Write $\Delta_{M,>0}=\big\{(k_1,\dots,k_M)\in \mathbb{Z}^M:k_1\ge k_2\ge \dots \ge k_M>0 \big\}$. Since $\mathbf{k}\in \Delta_0$, there exists $M$ such that $\mathbf{k}=\left(k_1,\dots,k_M,0,0,0,\dots\right)$, with $\left(k_1,\dots,k_M\right)\in \Delta_{M,>0}$. Moreover, from (\ref{LemAssumption}) there exists $N_0$ large enough so that for all $N\ge N_0$, $\mathbf{k}^{(N)}$ is of the form:
\begin{align*}
\mathbf{k}^{(N)}=\left(k_1^{(N)},\dots, k_M^{(N)},0,\dots,0,\star,\star,\star,\dots\right), \ \textnormal{ with } \left(k_1^{(N)},\dots, k_M^{(N)}\right)\in \Delta_{M,>0},
\end{align*}
where the $\star$'s are arbitrary negative integers. Thus, by restricting to $\big\{\left(k_1^{(N)},\dots, k_M^{(N)}\right)\big\}_{N\ge N_0}$ the claim (\ref{claim1}) follows immediately from the, easy to see, fact that if:
\begin{align*}
\Delta_{M,>0}\ni\left(k_1^{(N)},\dots, k_M^{(N)}\right)\longrightarrow \left(k_1,\dots, k_M\right)\in\Delta_{M,>0}
\end{align*}
then we have:
\begin{align*}
\big\{ l_i\left(\mathbf{k}^{(N)}\right)\big\}_{i=1}^{\infty}\overset{N\to \infty}{\longrightarrow}\big\{ l_i\left(\mathbf{k}\right)\big\}_{i=1}^\infty.
\end{align*}
Claim (\ref{claim2}) also follows similarly.
\end{proof}

We finally put all the pieces together to prove Theorem \ref{MainResult}.

\begin{proof}[Proof of Theorem \ref{MainResult}]
Observe that, since from Proposition \ref{DeltaNaughtProp} $\mathbf{k}(\mathbf{Z})\in \Delta_0$ for $\mathsf{M}^{(s)}$-a.e. $\mathbf{Z}$, we have that $\{k_i(\mathbf{Z})\}_{i=1}^\infty$ is completely determined by $\{l_i(\mathbf{Z})\}_{i=1}^{\infty}$ which is in turn determined by $\big\{\sum_{j=i}^{\infty}l_j(\mathbf{Z}) \big\}_{i=1}^\infty$. Thus, by combining Propositions \ref{PropositionApproximation} and \ref{DeltaNaughtProp} and Lemma \ref{DeterministicConvergenceLemma} it remains to take the $N\to\infty$ limit of $\big\{\sum_{j=i}^{\infty}l^{(N)}_j(\mathbf{Z}) \big\}_{i=1}^\infty$. But by Theorems \ref{MarkovChainRep1} and \ref{MarkovChainRep2} we have that for each $N\ge 1$ the random sequence $\big\{\mathsf{X}_i^{(N)}\big\}_{i=1}^\infty$, where $\mathsf{X}_i^{(N)}=\sum_{j=i}^{\infty}l^{(N)}_j(\mathbf{Z})$, is distributed as a time homogeneous Markov chain with initial distribution $\pi_N^{(s)}(N-\cdot)$ and transition kernel $\mathsf{P}^{(s)}$. The statement of the theorem then follows by making use of Lemma \ref{PosNoConvLem}.
\end{proof}

\section{Appendix: Proof of Proposition \ref{PropositionApproximation}}\label{Appendix}
In order to prove Proposition \ref{PropositionApproximation} we require some preliminaries. We first need the notion of an orbital measure. For $\mathbf{k}^{(N)}\in \Delta_N$ define the probability measure $\mathsf{Orb}_{\mathbf{k}^{(N)}}$:
\begin{align*}
\mathsf{Orb}_{\mathbf{k}^{(N)}}=\mathsf{Law}\left(\mathbf{B}\cdot\textnormal{diag}\left(p^{-k^{(N)}_1},p^{-k^{(N)}_2},\dots,p^{-k^{(N)}_N}\right)\cdot \mathbf{C}\right),
\end{align*} 
where $\mathbf{B}$ and $\mathbf{C}$ are independent Haar distributed matrices from $\textnormal{GL}(N,\mathbb{Z}_p)$. By identifying $\textnormal{Mat}\left(N,\mathbb{Q}_p\right)$ in a natural way with the subset of $\textnormal{Mat}\left(\mathbb{N},\mathbb{Q}_p\right)$ we view $\mathsf{Orb}_{\mathbf{k}^{(N)}}$ as a probability measure on  $\textnormal{Mat}\left(\mathbb{N},\mathbb{Q}_p\right)$ which by definition is $\textnormal{GL}(\infty,\mathbb{Z}_p)\times \textnormal{GL}(\infty,\mathbb{Z}_p)$-invariant. We then have the following proposition on convergence of orbital measures.

\begin{prop}\label{OrbitalMeasuresProp}
Let $\mathbf{k}\in \Delta$ and $\{\mathbf{k}^{(N)}\}_{N\ge 1}$ with $\mathbf{k}^{(N)}\in \Delta_N$. Then, the sequence of probability measures $\big\{\mathsf{Orb}_{\mathbf{k}^{(N)}} \big\}_{N\ge 1}$ converges weakly to some probability measure $\mu$ on $\textnormal{Mat}\left(\mathbb{N},\mathbb{Q}_p\right)$:
\begin{align}\label{OrbitalLimits1}
\mathsf{Orb}_{\mathbf{k}^{(N)}}\implies \mu
\end{align}
if and only if:
\begin{align}\label{OrbitalLimits2}
\mathbf{k}^{(N)} \longrightarrow \mathbf{k}.
\end{align}
If that is the case then $\mu=\mu_{\mathbf{k}}$ defined in Definition \ref{DefinitionErgodicMeasure}.
\end{prop}

\begin{proof}
We first prove $(\ref{OrbitalLimits2})\implies(\ref{OrbitalLimits1})$. From (\ref{OrbitalLimits2}) we get that $\big\{\mathsf{Orb}_{\mathbf{k}^{(N)}} \big\}_{N\ge 1}$ is tight, so every subsequence has a convergent subsubsequence. It then suffices, by the subsequence principle, to prove that all of these limits coincide and equal $\mu_{\mathbf{k}}$ which is then part of the proof of Theorem 8.8 in \cite{BufetovQiuClassificaiton}.

Conversely suppose (\ref{OrbitalLimits1}) holds. By Lemma 8.6 of \cite{BufetovQiuClassificaiton} we get that $\sup_{N\ge 1}k_1^{(N)}<\infty$. Then, every subsequence $\{\mathbf{k}^{(n_l)}\}_{l\in \mathbb{N}}$ has a convergent subsubsequence $\{\mathbf{k}^{(\tilde{n}_l)}\}_{l\in \mathbb{N}}$ to some $\mathbf{k}\in \Delta$. Namely, as in Definition \ref{DefinitionSeqConv}, for any $j\in \mathbb{N}$ there exists $k_j \in \mathbb{Z}\cup \{-\infty \}$ so that:
\begin{align*}
\lim_{l\to \infty}k_j^{(\tilde{n}_l)}=k_j.
\end{align*}
By the subsequence principle it again suffices to prove that all these limits coincide. Suppose otherwise, that there exists sequences $\{m_l \}_{l\in \mathbb{N}}, \{m'_l \}_{l\in \mathbb{N}}$ and $\mathbf{k}\neq \mathbf{k}' \in \Delta$ such that:
\begin{align*}
\mathbf{k}^{(m_l)}\overset{l\to \infty}{\longrightarrow}\mathbf{k}, \ \mathbf{k}^{(m'_l)}\overset{l\to \infty}{\longrightarrow}\mathbf{k}'.
\end{align*}
Then, from the implication $(\ref{OrbitalLimits2})\implies(\ref{OrbitalLimits1})$, we get that:
\begin{align*}
\mathsf{Orb}_{\mathbf{k}^{(m_l)}}\overset{l\to \infty}{\implies} \mu_{\mathbf{k}}, \ \mathsf{Orb}_{\mathbf{k}^{(m'_l)}}\overset{l\to \infty}{\implies} \mu_{\mathbf{k}'}.
\end{align*}
But by Proposition 5.1 of \cite{BufetovQiuClassificaiton} $\mu_{\mathbf{k}}$ and $\mu_{\mathbf{k}'}$ are distinct which contradicts the weak convergence in (\ref{OrbitalLimits1}).
\end{proof}

Finally, Proposition \ref{PropositionApproximation} follows from a combination of Propositions \ref{ErgodicProp1} and \ref{ErgodicProp2} below.

\begin{prop}\label{ErgodicProp1}
Let $\mathcal{M}\in \mathcal{P}_{\textnormal{inv}}\left(\textnormal{Mat}\left(\mathbb{N},\mathbb{Q}_p\right)\right)$. Then, $\mathcal{M}$ is supported on $\textnormal{Mat}_{\textnormal{reg}}\left(\mathbb{N},\mathbb{Q}_p\right)$.
\end{prop}

We define
\begin{align*}
\mathsf{Sing}:\textnormal{Mat}_{\textnormal{reg}}\left(\mathbb{N},\mathbb{Q}_p\right) &\longrightarrow \Delta,\\
\mathbf{Z} &\mapsto \left(k_1\left(\mathbf{Z}\right),k_2\left(\mathbf{Z}\right),k_3\left(\mathbf{Z}\right),\dots\right).
\end{align*}
and observe that $\mathsf{Sing}$ is a Borel map.

\begin{prop}\label{ErgodicProp2}
Let $\mathcal{M}\in \mathcal{P}_{\textnormal{inv}}\left(\textnormal{Mat}\left(\mathbb{N},\mathbb{Q}_p\right)\right)$ and let $\nu^{\mathcal{M}}$ be the corresponding ergodic decomposition probability measure on $\Delta$. Then,
\begin{align*}
\left(\mathsf{Sing}\right)_*\mathcal{M}=\nu^{\mathcal{M}}.
\end{align*}
\end{prop}

\begin{proof}[Proof of Proposition \ref{ErgodicProp1}]
First, let $\mu_{\mathbf{k}}\in \mathcal{P}_{\textnormal{erg}}\left(\textnormal{Mat}\left(\mathbb{N},\mathbb{Q}_p\right)\right)$. By Vershik's ergodic theorem, see Theorem 3.2 in \cite{OlshanskiVershik}, $\mu_{\mathbf{k}}$ is concentrated on the set of $\mathbf{Z}\in \textnormal{Mat}\left(\mathbb{N},\mathbb{Q}_p\right)$ for which the orbital measures $\big\{\mathsf{Orb}_{\mathbf{k}^{(N)}\left(\mathbf{Z}\right)}\big\}_{N\ge 1}$ weakly converge to $\mu_{\mathbf{k}}$. By Proposition \ref{OrbitalMeasuresProp} this set consists of those matrices $\mathbf{Z}$, so that the limits (\ref{OrbitalLimits2}) exist and coincide with the parameters $\mathbf{k}$ of $\mu_{\mathbf{k}}$. All such $\mathbf{Z}$ by definition belong to $\textnormal{Mat}_{\textnormal{reg}}\left(\mathbb{N},\mathbb{Q}_p\right)$ and thus $\mu_{\mathbf{k}}$ is supported on $\textnormal{Mat}_{\textnormal{reg}}\left(\mathbb{N},\mathbb{Q}_p\right)$. Now, suppose $\mathcal{M} \in \mathcal{P}_{\textnormal{inv}}\left(\textnormal{Mat}\left(\mathbb{N},\mathbb{Q}_p\right)\right)$ with ergodic decomposition measure $\nu^{\mathcal{M}}$.
Let $\mathfrak{F}(\cdot)=\mathbf{1}_{\textnormal{Mat}_{\textnormal{reg}}\left(\mathbb{N},\mathbb{Q}_p\right)}\left(\cdot\right)$ and apply Proposition \ref{ErgodicDecompositionProp}. Since $\mu_{\mathbf{k}}\left(\mathfrak{F}\right)=1$ and $\nu^{\mathcal{M}}$ is a probability measure we get $\mathcal{M}\left(\mathfrak{F}\right)=1$. The proposition is fully proven.
\end{proof}

\begin{proof}[Proof of Proposition \ref{ErgodicProp2}]
Let $\mathsf{F}$ be an arbitrary bounded Borel function $\Delta$ and $\mathfrak{F}$ be its pullback on $\textnormal{Mat}_{\textnormal{reg}}\left(\mathbb{N},\mathbb{Q}_p\right)$. We then need to show:
\begin{align*}
\mathcal{M}\left(\mathfrak{F}\right)=\nu^{\mathcal{M}}\left(\mathsf{F}\right).
\end{align*}
Firstly, recall that from the proof of Proposition \ref{ErgodicProp1} $\mu_{\mathbf{k}}$ is supported on $\mathsf{Sing}^{-1}\left(\mathbf{k}\right)\subset \textnormal{Mat}_{\textnormal{reg}}\left(\mathbb{N},\mathbb{Q}_p\right)$ and by the very definition of $\mathfrak{F}$, $\mathfrak{F}\big|_{\mathsf{Sing}^{-1}\left(\mathbf{k}\right)}=\mathsf{F}(\mathbf{k})$ so that $\mu_{\mathbf{k}}\left(\mathfrak{F}\right)=\mathsf{F}(\mathbf{k})$.
Then, by definition we have:
\begin{align*}
\mathcal{M}\left(\mathfrak{F}\right)=\int_{\Delta}^{}\mu_{\mathbf{k}}\left(\mathfrak{F}\right)\nu^{\mathcal{M}}\left(d\mathbf{k}\right)=\int_{\Delta}^{}\mathsf{F}(\mathbf{k})\nu^{\mathcal{M}}\left(d\mathbf{k}\right)
\end{align*}
which is exactly what we wanted to prove.
\end{proof}

\bigskip
\noindent
{\sc School of Mathematics, University of Edinburgh, James Clerk Maxwell Building, Peter Guthrie Tait Rd, Edinburgh EH9 3FD, U.K.}\newline
\href{mailto:theo.assiotis@ed.ac.uk}{\small theo.assiotis@ed.ac.uk}

\end{document}